\documentclass{article}
\usepackage{amssymb, amsthm, amsmath, tikz, comment}
\newtheorem{Lem}{Lemma}
\newtheorem{Claim}[Lem]{Claim}
\newtheorem{Def}{Definition}
\newtheorem{thm}{Theorem}
\newtheorem{Conj}{Conjecture}
\newcommand{\vol}{\mathop{\mathrm{vol}}\nolimits}
\newcommand*{\covol}[2][\mathbb{R}^d]{\vol\left( #1/#2 \right)}
\newcommand{\dist}{\mathop{\mathrm{dist}}\nolimits}
\newcommand{\gen}{\mathop{\mathrm{span}}\limits}
\newcommand{\km}{\kappa_{\max}^{-1}}

\begin{document}
\title{Tail asymptotics of free path lengths for the periodic Lorentz process.\\ On Dettmann's geometric conjectures.}
\author{P\'{e}ter N\'{a}ndori, Domokos Sz\'{a}sz, and Tam\'{a}s Varj\'{u}
\footnote{Institute of Mathematics, Budapest University of Technology and
Economics}
\footnote{ email: nandori@math.bme.hu, szasz@math.bme.hu,
kanya@math.bme.hu %\\
%e-mail: nandori,szasz,kanya math.bme.hu
}}
\vskip2mm

\maketitle
\date{}
\vskip2mm

\begin{abstract}

 In the simplest case, consider a $\mathbb{Z}^d$-periodic ($d \geq 3$)
 arrangement of balls of radii $< 1/2$, and select a random direction
 and point (outside the balls). According to Dettmann's first
 conjecture, the probability that the so determined free
 flight (until
 the first hitting of a ball) is larger than $t >>1$ is  $\sim \frac{C}{t}$,
 where $C$ is
 explicitly given by the geometry of the model. In its simplest form,
 Dettmann's second conjecture is related to the previous case with
 tangent balls (of radii $1/2$). The conjectures are established in a
 more general setup: for $\mathcal{L}$-periodic configuration of - possibly intersecting - convex bodies
 with $\mathcal{L}$ being a non-degenerate lattice. These questions are related
 to P\'olya's visibility problem (1918), to theories of
 Bourgain-Golse-Wennberg (1998-) and of
 Marklof-Str\"{o}mbergsson (2010-). The results
 also provide the asymptotic
 covariance of the periodic Lorentz process assuming it has a limit
 in the super-diffusive scaling, a fact if $d = 2$ and the horizon is
 infinite.
\end{abstract}
\section{Introduction}
\label{sec:intro}
The subject of our paper is the verification of the first two, purely geometric, conjectures of Dettmann, circumscribed in the abstract (the final, third one is of dynamical feature). A substantial motivation for the conjectures - and for us, too - came from the dynamical theory  of Brownian motion, more concretely from that of the periodic Lorentz process. Therefore the introduction will consist of two parts. In the first one we restrict ourselves to the geometric problems, whereas in the second one we treat the motivation coming from the dynamical theory of the periodic Lorentz process. (We also hope that our forecast for the final picture of the limit theorems of the dynamical theory will speed up filling out the missing details.)

\subsection {The geometric conjectures}
 In the simplest case, consider a $\mathbb{Z}^d$-periodic ($d \geq 3$)
 arrangement of balls of radii $r< 1/2$, and select a random direction
 and point (outside the balls). According to Dettmann's first
 conjecture,
 \begin{enumerate}
\item
 the probability that the so determined free
 flight $\tau_r$ (until
 the first hitting of a ball) is larger than $t >>1$ is  $\sim \frac{C_r}{t}$, and
 \item the constant $C_r$ is
 explicitly given by the geometry of the model.
 \end{enumerate}

 \cite{BGW98,GW00}  provided upper and lower bounds for the aforementioned tail probability, whereas \cite{MS10} gave precise description of the rescaled tail distribution of $r^{d-1} \tau_r$ when $r \to 0$. This latter is the well-known Boltzmann-Grad limit of statistical physics when the average length of the free path tends to a constant. (Both teams observed the surprising phenomenon that the Boltzmann-Grad limit of the Lorentz-process, i. e. of the billiard dynamics is not the classical linear Boltzmann equation.) The answer to P\'olya's 1918 visibility problem (cf. \cite{K08}) is a simple consequence of the results above.

  In its simplest form,
 Dettmann's second conjecture is related to the case with
 tangent balls (of radii $1/2$).  In this case, the expected asymptotics is $\asymp t^{-2}$ if $d < 6$ (with a logarithmic correction if $d=6$) and is $\asymp t^{-\alpha_d}$ if $d > 6$ where $1< \alpha_d < 2$.

 The main goal of our work is the proof of these two conjectures. In fact, we establish them in a much more general setup for any $\mathcal{L}$-periodic configuration of - possibly intersecting - convex bodies
 with $\mathcal{L}$ being a non-degenerate lattice. We also emphasize that here $r$ is fixed and does not tend to zero. Moreover, in case of the second conjecture we also provide the exact values of the exponents $\alpha_d$. We note that the essential mathematical difficulties of both proofs are already present in the aforementioned simplest cases.

 \subsection{Motivation: The dynamical problem}

% solOur subject is a toric billiard with strictly convex smooth scatterers (i. e. the celebrated Sinai billiard) and its
In 1905, Hendrik Lorentz \cite{L05} introduced Lorentz gas as a model of motion of classical electrons in a metal. The (periodic) \emph{Lorentz process} is the dynamics of just one electron in a crystal. It is the $\mathbb Z^d$-extension
of a \emph{toric Sinai billiard} (i.~e.~of one with strictly convex
smooth scatterers on the $d$-torus, $d \ge 2$ ). Unfortunately stochastic properties of Sinai-billiards - and more generally of semi-dispersing ones - have been established in the planar case, only. Nevertheless, if the dynamical theory of these billiards will prove those properties as expected, then  our results will also 1.) forecast when exactly one has super-diffusive scaling rather than diffusive one and, moreover, 2.) provide the asymptotic covariance under the super-diffusive scaling.

Let us explain the previous ideas in more detail. It is known that for
the \emph{planar Lorentz process} the limiting distribution of the
rescaled displacement is Gaussian and that of the rescaled orbit is a
Wiener process. The scaling, however, is either the diffusive $\sqrt
n$ or the slightly super-diffusive $\sqrt {n \log n}$ depending on
whether the billiard has finite or infinite \emph{horizon}, resp.~(we
say that the horizon is finite if the free flight time is finite). In the first case the
limiting covariance is given by the Green-Kubo formula (cf. \cite{BS81}, \cite{BSCh91}), which - though
explicit - nevertheless does not permit precise calculations
(the formula contains an infinite sum of time correlations of the free
flight vector). In the infinite horizon case, however, - as it was
conjectured by \cite{B92} and established by \cite{SzV07}, \cite{ChD09} - the  stronger $\sqrt {n \log n}$
scaling suppresses time correlations and
the limiting covariance has a simple form expressed by geometric
parameters of the billiard in question.

For multidimensional Sinai billiards - even under the complexity hypothesis, expected to hold typically
- exponential decay of correlations is known in the finite horizon
case, only (cf.~\cite{BT08}). Then the central limit theorem with the
diffusive scaling is a consequence and the limiting covariance is
again given by the Green-Kubo formula.  Physicists are always emphatically
interested in expressions that are easy to calculate and check.
Dettmann, \cite{D12}, motivated by a problem of \cite{Sz08} and the most precious - computationally supported - observations of \cite{S08}, was assuming that
the aforementioned 2D infinite horizon case picture is also valid for
multidimensional dispersing billiards and made a guess as to how the
limiting covariance looks like. The difficulty is that, in this case,
the structure of the horizons, i.~e.~orbits which never meet any
scatterer, is much more complicated than in the planar case.

In fact, Dettmann formulates three conjectures for $\mathbb
Z^d$-periodic Lorentz processes. The first two make claims for the
tail asymptotic of the free path length. Roughly speaking the first
one is related to the generic cases whereas the second one to certain
degenerate cases. (In both cases a Wiener limit is expected with
diffusive or super-diffusive scaling.) These conjectures are of
purely geometric nature and the main goal of our work is to
establish them. We do this in a wider generality: 1) for semi-dispersing
billiards, 2) possibly with corner points, 3) and permitting arbitrary
lattices \(\mathcal{L}\) of finite covolume rather than only $\mathbb
Z^d$. By accepting the dynamical hypothesis that the multidimensional
picture is analogous to the 2D one (i.~e.~ {\bf A.} there is an exponential
decay of cross correlations, and {\bf B.} whether there is super-diffusive or
diffusive behavior only depends on the tail asymptotic of the free
path length), the first conjecture, among others, implies that - similarly to the
planar case - the super-diffusivity covariance has a simple form that
can be calculated from the geometry of the billiard. The second
conjecture supports the hypothesis that, indeed,  degenerate
billiards, i.~e.~ those without an open configuration subset of collision-free subspaces of maximal dimension $d-1$,  always have diffusive behavior. It is worth mentioning that our Theorem \ref{tetel2} also provides exact values of the exponents $\alpha_d$ in cases $d > 6$ where Dettmann only guessed $1 < \alpha_d < 2$.  Dettmann's third conjecture, also a dynamical one,
supports the previously mentioned dynamical hypothesis since it is about sufficiently strong decay of correlations being the subject of future progress of the theory.

The paper is structured as follows. In Section \ref{sec:results}, we
provide the definitions and formulate Dettmann's conjectures together
with our results. In Section \ref{sec:prep}, we prove some finiteness
lemmas and introduce an important tool which is the fattening of the
configuration space (or shrinking of the scatterers, in other words).
A key lemma to our proof of Theorem 1 is the so-called Proportionality lemma,
which we discuss in Section \ref{sec:tail}. Sections \ref{sec:proof}
and \ref{sec:inc} are devoted to the proofs of Dettmann's first
and second conjectures, respectively (it is worth noting that their methods are completely different). In Section \ref{sec:ex},
we present instructive examples where the
super-diffusive limiting covariance matrix can be calculated: one of them is
the first multidimensional semi-dispersing billiard whose ergodicity
got ever proved: a three-dimensional toric billiard with two
cylindrical scatterers (cf.~\cite{KSSz89}). The second one is the model of two hard balls on $\mathbb T^d: d \ge 3$.
Finally, we make some concluding remarks in Section
\ref{sec:rem}. In particular we also describe briefly the relation of our setup to that of Bourgain-Golse-Wennberg, \cite{BGW98} and of
 Marklof-Str\"{o}mbergsson, \cite{MS10}.

\section{Setup and main results}
\label{sec:results}

\subsection{$\mathcal L$-periodicity and the dynamics}

\paragraph{Periodicity}
We consider an infinite configuration space \(\tilde Q\subset
\mathbb{R}^d\) and a lattice (i.~e.~ a discrete additive subgroup \(\mathcal{L}\subset
 \mathbb{R}^d\) of finite covolume)    defining the periodicity
of the Lorentz gas.  Assuming that the configuration space is invariant
under translations in \(\mathcal{L}\)  one can also consider the
compact configuration space \(Q=\tilde Q/\mathcal{L}\). For latter reference we recall that a linear subspace is called a \emph{lattice subspace} if it can be
  generated by lattice vectors.

\paragraph{Scatterers}
The complement of the compact configuration space consists of finitely
many open, convex sets \(\mathbb{R}^d/\mathcal{L} \setminus Q=
\bigcup_{i=1}^n\mathcal{O}_i\) (called scatterers, or obstacles).
Equivalently \(\mathbb{R}^d\setminus \tilde Q = \bigcup_{i=1}^n
\bigcup_{l\in\mathcal{L}} (\mathcal{O}_i + l)\).  We assume that the
boundary of each $O_i$ is a \(\mathcal{C}^3\)-smooth hypersurface.

Notice that we do not require the scatterers to be disjoint, nota bene
different scatterer configurations can lead to identical configuration
spaces, if the differences are covered by other scatterers.  Points in
the boundary intersections \(q \in \partial \mathcal{O}_i
\cap \partial \mathcal{O}_j\) are called corner points.

\paragraph{Curvature upper bound}
It is required that, at any point of the boundary $\partial Q$, the
curvature operator \(K\) is uniformly bounded from above: there exists
a universal constant \(\kappa_{\max}\), such that for every
tangent vector $v$ of the hypersurface $\partial Q$, the inequality
\(0\leq K(v,v)\leq\kappa_{\max}\|v\|^2\) holds.

\paragraph{Dynamics and phase space}
The continuous time dynamics \(\Phi_t\) acts on the phase space
\(\tilde M=\tilde Q\times \mathbb{R}^d/\sim\), where \(\sim\) is the
identification of pre-collisional and post-collisional velocities on
\(\partial \tilde Q\), which are mirror images with respect to the
tangent hyperplane of the boundary at that point.
We also write $\Phi_{[t_1,t_2]} x$ for the set
$\{\Phi_s x | t_1 \leq s \leq t_2 \}$.
For later
definitions and statements if we write \(x=(q,v)\) with \(q
\in \partial \tilde Q\), then \(v\) is chosen as the post collisional
one.  At corner points there are more than one such hyperplanes, and
mirroring generally does not commute, so the dynamics is either not
defined, or has multiple values.  (Since the speed \(|v|\) is
invariant under the dynamics, in the literature one usually takes the
phase space \(\tilde M=\tilde Q\times {S}^{d-1}/\sim\). It will not lead to a contradiction that for us often it will be more convenient
to consider $\tilde M$ as introduced
above.)

The action is free flight \(\Phi_t(q,v) = (q+tv,v)\) as long as \(q+vt
\not \in \partial \tilde Q\).  On the boundary the velocity is reset
to the post collisional one, and free flight follows with that vector.
Moreover, even if the orbit hits a scatterer and the collision is tangent
(sometimes called grazing), the dynamics is still free flight since in this case
the velocity does not change.  The
dynamics is invariant under \(\mathcal{L}\)-translations, so the
compact phase space of the flow is \(M =
Q\times {S}^{d-1}/\sim\). For simplicity, we will use the same
notation \(\Phi_t\) for the flow on the compact phase space as well.
The Lorentz dynamics has natural invariant measures, the
Liouville-ones: \(d\mu = \textrm{const.} dq dv$ on $\tilde M\).  The
\(\textrm{const.}=1\) measure is called Lebesgue.  Similarly the
invariant probability measure for the billiard dynamics on $M$ is
$d\mu = c_\mu dq dv$ with $c_\mu = (\vol Q \vol
{S}^{d-1})^{-1}$. We will also use the notation $\lambda_{d'}$
for the Lebesgue measure in dimension $d' \leq d$.

\paragraph{Billiard and Lorentz process}
\begin{Def}
  Under the aforementioned conditions, the dynamics \(\Phi_t (t \in
  \mathbb R)\) on the phase space $M$ is called a
  \emph{semi-dispersing billiard} and that on the phase space $\tilde
  M$ a (semi-dispersing) \emph{Lorentz process}. If the scatterers are
  strictly convex, then the billiard is called a \emph{dispersing} one
  or a \emph{Sinai-billiard}.
\end{Def}

In this paper we will consider a fixed semi-dispersing billiard (or
the corresponding Lorentz process) satisfying the aforementioned
conditions.

\paragraph{The free flight function}
For \(x=(q,v)\) \[\tau(x) = \inf \{t>0 \mid q+tv \in \cup_i
\mathcal{O}_i\}\] as usual, the infimum of the empty set is
\(\infty\).  This definition is slightly  different from the
usual definition.  In fact, at points where the first collision is
tangential, the new definition gives a larger value.  The advantage of
this change is seen by the semi-continuity Claim \ref{lem:tau_usc}.  It
is obviously invariant under \(\mathcal{L}\) translations, so we will
not distinguish whether the function is defined on the compact or on
the non-compact phase space.

Our main focus will be on the tail distribution of the free path
length:
\begin{equation}
\label{eq:F}
F(t)=\mu(\tau>t)
\end{equation}
 i.~e.~of the probability of surviving
without collision for time \(t\).

\subsection{Horizons}

\begin{Def}
  For a configuration point \(q\in\tilde Q\) a {\rm free subspace}
  \(V\) is a maximal (for containment) linear subspace of
  \(\mathbb{R}^d\), such that \(q+V\subset\tilde Q\).  This latter is
  equivalent to requiring \(\tau(q+v,w) = \infty \) for all \(v, w \in
  V\). (Sometimes we also call the affine subspace $q + V$ a free subspace.)
\end{Def}
\begin{Claim}
\label{lem:free_subs_lattice}
  Any free subspace \(V\) is a lattice subspace.
\end{Claim}
\begin{proof}
  If we have a vector \(v\in V\), then by invariance \(q + tv + l \in
  \tilde Q\) for all \(t\in\mathbb{R}\) and \(l\in\mathcal{L}\).  If
  this vector \(v\) is not parallel to a lattice vector, then the set
  \(tv+l\) is dense in some lattice subspace \(V'\), concluding \(q+V'
  \subset \tilde Q\), so \(V'\subset V\) by maximality.
\end{proof}

\begin{Def}
  If \(\tau(q,v)=\infty\), consider a free subspace \(V\ni v\) for
  \(q\), and the subspace \(V^{\perp}\) orthogonal to it.  A maximal
  connected subset \(\tilde B_H \subset \tilde Q \cap (q+V^\perp)\)
  containing points \(q'\), for all of which \(V\) is a free subspace
  is called a {\rm basis for the horizon} \(\tilde H = \tilde B_H
  \times V \subset \tilde Q\). The
  {\rm dimension \(d_{\tilde H}\) of the horizon}  is the dimension of the free
  subspace \(V\).
  \end{Def}

\paragraph{Remark} Of course, for one and the same horizon the set of
possible bases is invariant under $V$-shifts. If we talk about the basis
$\tilde B_H$ of a horizon $\tilde H$, then we think of $\tilde B_H$ as represented in $\tilde Q$ for an arbitrary
$q \in \tilde H$.

\begin{Def}
By taking $Q = \tilde Q/\mathcal L$ one obtains  the {\rm horizon} $H = \tilde H/\mathcal L$ in the compact configuration space. Its {\rm basis $B_H$} is
$\tilde B_H$ as represented in $Q$, and its {\rm dimension $d_H$} coincides with \(d_{\tilde H}\).
\end{Def}

\paragraph{Remark} At factorization, for a $q \in Q$,  $q+V^\perp$ can contain several copies of the basis $B_H$.

\paragraph{Remark}
We will use the same names, though different notations, for the
  phase space counterpart  \(\tilde{\mathcal{H}} = \tilde {H} \times V
  \subset \tilde M\) of a horizon $\tilde H$ and for the corresponding sets \(H, \mathcal{H}\)
  in the compact spaces.  Also, to denote the relation between the horizon and its
  velocity component, we will sometimes use the notation \(V_H\).
\bigskip

A substantial observation of Dettmann is that $F(t)$ (see (\ref{eq:F})) can asymptotically be expressed as a sum for a finite number of horizons $H$ of times a free flight spends in $H$. Concretely,
he introduced the probability of remaining
within a horizon \(H\) for time $t$,
that is \[F_H(t) = \mu \left( \left\{ (q,v) \in
    M \mid q+sv \in H, \enskip \forall s \in [0,t] \right\} \right)\]
a quantity that can be calculated exactly. (cf.~Equ.~(26) of \cite{D12}
or (\ref{eq:F_H}) to be given later).

\begin{Def} (\cite{D12})
\begin{itemize}
\item A \emph{maximal horizon} is one of the highest dimension for the
  given billiard (or Lorentz process).
\item A \emph{principal horizon} is one of the highest dimension
  possible, which is \(d-1\) if there are scatterers.
\item A horizon \(H\) is \emph{incipient} if its basis $B_H$ has
  (\(d-d_H\) dimensional) measure zero.
\end{itemize}
\end{Def}

Denote the set of maximal non-incipient horizons by $\mathbb
H$.  It can be empty if all maximal horizons are incipient, or there
are no horizons at all.

We conclude this point with a simple lemma.

\begin{Lem}\label{lem:horozon shape}
The boundary of the basis of a horizon consists of $\mathcal C^3$, concave  pieces except for principal horizons when it consists of two endpoints of an interval (they may coincide).
\end{Lem}

\subsection{Dettmann's conjectures, \cite{D12}}

\begin{Conj}
  Consider an $\mathcal L$-periodic Lorentz process with at least one
  non-incipient maximal horizon. Then, as $t \to \infty$ we have
  \[
  F(t) \sim \sum_{H \in \mathbb H} F_H(t).
  \]
\end{Conj}

\begin{Conj}
  Consider an $\mathcal L$-periodic Lorentz process with incipient
  (but no actual) principal horizon. Then, as $t \to \infty$, we have
  \[
  F(t)\asymp
  \begin{cases}
    t^{-2},  & d < 6\\
    t^{-2} \log t,  & d = 6\\
    t^{-\alpha_d} \qquad (1 < \alpha_d < 2), & d > 6
  \end{cases}
  \]
\end{Conj}

These two conjectures are of purely geometric nature, whereas
the following one concerns the dynamics, too.

\begin{Conj}
  Consider an $\mathcal L$-periodic Lorentz process and let $f, g : M
  \to \mathbb R$ denote zero-mean (wrt the invariant measure $\mu$)
  H\"older functions. Then, as $t \to \infty$, we have
  \[
  \int_{\{x \in M \mid \tau(x)<t\}} (f) (g \circ \Phi_t) d \mu =
  o(F(t)).
  \]
\end{Conj}

\subsection{Main results}
Now we can formulate the main results of our paper.

\begin{thm}
\label{tetel1}
  Consider an $\mathcal L$-periodic semi-dispersing Lorentz process
  (possibly with corner points). Assume it has at least one
  non-incipient maximal horizon.  Then, as $t \to \infty$ we have
  \[
  F(t) \sim \sum_{H \in \mathbb H} F_H(t).
  \]
\end{thm}
\begin{thm}
\label{tetel2}
  Consider an $\mathcal L$-periodic semi-dispersing Lorentz process
  (possibly with corner points). Assume it has at least one incipient
  (but no actual) principal horizon. Then, as $t \to \infty$, we have
  \[
  F(t)=
  \begin{cases}
    O(t^{-2}), & 3 \leq d \leq 5\\
    O(t^{-2} \log t),  & d = 6\\
    O\left(t^\frac{2+d}{2-d}\right), & d > 6.
  \end{cases}
  \]
  Further, if we also assume that the curvature is bounded away from $0$ (from below) uniformly
  at every point of $\partial Q$ (dispersing case), then
   \[
     F(t) \asymp
       \begin{cases}
           t^{-2}, & 3 \leq d \leq 5 \\
           t^{-2} \log t,  & d = 6\\
	   t^\frac{2+d}{2-d}, & d > 6.
	         \end{cases}
		   \]

\end{thm}

\noindent{\bf Remark.} According to the dynamical theory of semi-dispersing billiards super-diffusive behavior can only arise if the asymptotics of $F(t)$ is non-integrable. Therefore Theorems \ref{tetel1}, \ref{tetel2} and (\ref{eq:F_H})
suggest that, in the absence of principal, non-incipient horizon, no super-diffusive behavior is possible (cf. Section  \ref{sec:ex}). Moreover, in the case of super-diffusivity the scaling is $\sqrt{t \log t}$ - again by (\ref{eq:F_H}).

\section{The method of fattening, finiteness and stability lemmas}
\label{sec:prep}

\subsection{Lattice geometry}
\label{sec:lattice}
\label{sec:F_H(t)}
The following statement is well-known, in fact, quantitative results are also known, see for instance \cite{Sch68}.
\begin{Lem}
  For any \(K>0\) the number of lattice subspaces \(V\), such that
  \(\covol[V]{V\cap\mathcal{L}} < K\) is finite.
\end{Lem}

By this lemma the minimal covolume of \(k\) dimensional sublattices
exists, and we will denote it by \(\ell_k\).  For example \(\ell_1\)
is the minimal length of nonzero lattice vectors,
\(\ell_d=\covol{\mathcal{L}}\), and \(\ell_0=1\) as usual for empty
products.

\begin{Lem}\label{lem:prod_covol}
  If we have a lattice subspace \(V\), and we take its orthogonal
  complement \(V^\perp\), and we project \(\mathcal{L}\) orthogonally
  onto \(V^\perp\) to get \(\mathcal{L}_V^\perp\), then we
  have \[\covol{\mathcal{L}}=\covol[V]{V\cap\mathcal{L}}
  \covol[V^\perp]{\mathcal{L}_V^\perp}.\]
\end{Lem}
\begin{proof}
  Take a basis \(\{a_i\}_{i=1}^{\dim(V)}\) for \(\mathcal{L}\cap V\),
  and extend this to a basis \(\{a_i\}_{i=1}^d\) of \(\mathcal{L}\).
  Then \(|\det(a_i)|=\covol{\mathcal{L}}\).  The determinant does not
  change if we project the last \(d-\dim(V)\) vectors orthogonally to
  the orthocomplement of the first \(\dim(V)\) vectors.  The
  projections give rise to a basis of \(\mathcal{L}_V^\perp\), and by
  orthogonality \(|\det(a_i)|=|\wedge_{i=1}^{\dim(V)}
  a_i||\wedge_{i=\dim(V)+1}^d a_i^\perp|\), which is the claim.
\end{proof}

Now we can provide the asymptotic form of $F_H(t)$.
Indeed, in our notations, Equ.~(26) of \cite{D12} reads as
\begin{equation}
\label{eq:F_H}
F_H(t) \sim \frac{\vol{S_{d_H-1}} \int_{B_H} \int_{B_H}
\Delta_{B_H}^{\rm vis} (q, q') dq dq' }{(1 - \mathcal P)
\vol{S_{d-1}} \covol[V^\perp]{\mathcal{L}_V^\perp}}
\frac{1}{t^{d - d_H}} =:
C_H \frac{1}{t^{d - d_H}}
\end{equation}
where $\mathcal P = 1 - \frac{\vol{Q}}{\covol{\mathcal{L}}}$ is the volume fraction covered by scatterers and $\Delta_{B_H}^{\rm vis} (q, q')$ is the visibility function providing the number of possible connecting intervals $\overline{q, q'}$, lying in $B_H$, of the points $q, q'$ (toric geometry!). Note that the value of the integral is invariant under $V$-shifts of $B_H$ and is finite since
$\Delta_{B_H}^{\rm vis} (q, q')$ is bounded. So as to verify the latter,
assume by contradiction that for each $n>0$ one finds $q_n, q'_n$ such that
$\Delta_{B_H}^{\rm vis} (q_n, q'_n) >n$. Since the sets
$\Delta_n = \{ (q,q') | \Delta_{B_H}^{\rm vis} (q, q') >n\}$ are closed
subsets of each other, they have a nonempty intersection containing some
$(q_{\infty}, q'_{\infty})$ with $\Delta_{B_H}^{\rm vis}
(q_{\infty}, q'_{\infty}) = \infty$. Thus an infinite line is part of
$B_H$, which contradicts to its definition.

\medskip

\noindent{\bf Remark.} In the much interesting case of a {\it principal horizon $H$}, $B_H$ is an interval and the previous formula becomes simpler:
\begin{equation}
\label{eq:F_H2}
F_H(t) \sim \frac{2 \vol{S_{d-2}} |B_H|^2 }{(1 - \mathcal P) \vol{S_{d-1}} \covol[V^\perp]{\mathcal{L}_V^\perp}} \frac{1}{t}
\end{equation}

\subsection{Fattening, and its properties}
\label{sec:fattening}

The curvature upper bound implies in particular that, at any point of
the boundary, a tangent sphere of radius
\(\km\) is contained completely in the
scatterer.  This allows us to define the shrinking of the scatterers,
or equivalently the fattening of the configuration space by
\(0\le\delta<\km\) as a parallel domain (which
is typically not homothetic to the original one).  Indeed, define
\(\mathcal{O}_i^\delta\) as the centers of all balls of radius
\(\delta\), which are contained in \(\mathcal{O}_i\):
\[
\mathcal{O}_i^\delta = \{ q \in \mathcal{O}_i \mid \dist(q,\partial
\mathcal{O}_i) > \delta\}.
\]
This leads to new configuration spaces \(\tilde Q^\delta =
\mathbb{R}^d \setminus \cup_i\cup_{l \in \mathcal{L}}
(\mathcal{O}_i^\delta + l)\), and \(Q^\delta = \tilde Q^\delta /
\mathcal{L}\), which satisfy all the above assumptions, with
\((\km-\delta)^{-1}\) as a curvature upper
bound.

\paragraph{Upper semi-continuity}
The definition can be extended to negative values of $\delta$.  Also
note our previous comment that different scatterer configurations can
lead to the same configuration space.  Since fattening is defined from
scatterers, the same configuration space can have different fattenings
for the same \(\delta\).  The semigroup property of this operation
holds \(Q^0 = Q\), and \(\left( Q^\delta \right)^{\delta'} =
Q^{\delta+\delta'}\) as long as \(\delta, \delta'\) and
\(\delta+\delta'\) are all smaller than
\(\km\).  (By the latter restrictions this is
not exactly a semigroup.)

It is then natural to denote the corresponding dynamics by
\(\Phi^\delta_t\).  We denote by \(\tau^\delta\) the free flight
function on the fattened space.

\begin{Lem}[Upper semi-continuity]\label{lem:tau_usc}
  \(\tau^\delta\) as a function on
  $(-\infty,\km) \times \tilde M$ is upper
  semi-continuous (to be abbreviated as USC)  in all of its variables $(\delta, x)$.

Moreover, in the case of the previously defined fattening, the equality
\(\tau(x)=\tau^\varepsilon(x)\) holds with $\varepsilon> 0$  if and only if \(\tau(x)=\infty\).

Also, if for $(q, v) \in M$\ \  $\tau^\varepsilon (q, v) = \infty$ for every $\varepsilon > 0$, then $\tau (q, v) = \infty$, too.
\end{Lem}
\begin{proof}

  This only requires a proof at points \(x=(q,v)\), where
  \(\tau^\delta(x)<\infty\).  By the definition of \(\tau\), for a
  small \(\epsilon > 0\) we have \(q + (\tau^\delta(x) + \epsilon)v \in
  \mathcal{O}_i^\delta\) for some \(i\).  Since both the free flight
  dynamics and the fattening are continuous, we have that, for nearby points
  \(x'\), and nearby parameters \(\delta'\), $\exists \varepsilon > 0$ such that \(q' +
  (\tau^\delta(x) + \epsilon)v' \in \mathcal{O}_i^{\delta'}\).  The
  dynamics of a nearby point \(x'\) may differ from the free flight
  dynamics only $S^{[0, \tau^\delta(x)]}(x)$ had a jab (non tangent) collision 'before', but then
  \(\tau^{\delta'}(x')\) is even smaller than
  \(\tau^\delta(x)+\epsilon\).  (
  `before' permits equality as well thus the argument is also valid for
  simultaneous collisions at corner points.)
\end{proof}

\paragraph{Monotonicity} Of course, the fattening of the configuration
space makes free flights longer.  We will use it not only for the
above defined parallel domain, but for a larger set of inclusion
relations, too.

Denote by $\mathcal V^\varepsilon(q)$ (or by $\mathcal V(q)$) the set of free subspaces at $q \in \tilde Q^\varepsilon$ (or at $q \in \tilde Q$, respectively).

\begin{Lem}[\bf Monotonicity] \label{lem:mon}
 For $q \in Q$, $\mathcal V^\varepsilon(q)$ is an increasing function of $\varepsilon \in [0, \km)$ in the sense that $\forall 0 < \varepsilon < \varepsilon'$ and $\forall V\in \mathcal V^\varepsilon(q)$  $\exists V^* \in \mathcal V^{\varepsilon'}(q)$ such that $V \subset V^*$.
\end{Lem}

\begin{proof}
 If \(Q\subset Q'\), then for any
\(x\in M\) we have \(x\in M'\), too. Then we can consider both
free flights \(\tau\) and \(\tau'\) and we have
\(\tau(x) \leq \tau'(x)\).
\end{proof}

\paragraph{Local stability}
For $q \in \tilde Q$ and $\delta > 0$ denote by $B(q, \delta)$ the $\delta-$neighborhood of $q$. Unless specified otherwise, it will be considered as a neighborhood in  $\tilde Q$. (We use the same notation analogously for $q \in Q$.)

\begin{Lem}[\bf Local stability]\label{lem:stab}
For any \(q\in \tilde Q\) there exists
\(\xi > 0\) such that, for every \(q'\in B(q, \xi) \cap \tilde Q\) and any free subspace $V^\xi (q')$ for $\Phi^\xi$ at $q'$, there is a free subspace $V(q)$ for  $\Phi$ at $q$ such that $V^\xi (q')\subset V(q)$.

In other words, the set of free subspaces as a function of the base point $q \in Q$ and of $\varepsilon>0$ is upper semi-continuous at $q \in Q, \varepsilon = 0$ in the sense that for any $q_n \to q$ and $\varepsilon_n \searrow 0$ and for any $V \in \lim_{n \to \infty} \cup_{k \ge n} \mathcal V^{\varepsilon_k}(q_k)$ there is a $V^* \in \mathcal V(q)$ such that $ V \subset V^*$.
\end{Lem}

\begin{proof}
We prove the claim in its second form. Assume the contrary. Then there exists a velocity $v_\infty$ and sequences $q_n \to q$, $v_n \to v_\infty$ and $\varepsilon_n \to \varepsilon$ such that $\tau^\varepsilon(q, v_\infty) < \infty$
and $\tau^{\varepsilon_n}(q_n, v_n) = \infty$. This contradicts Lemma \ref{lem:tau_usc}.
\end{proof}

\subsection{ Finiteness of free subspaces}
\label{sec:finstab}

\begin{Lem}\label{lem:fin}
  For any configuration point \(q\in \tilde Q\) the set of free
  subspaces is finite.
\end{Lem}
\begin{proof}
  The proof is inductive by codimension \(d-\dim(V)\).  If  \(\dim V=d\), then there are
  no scatterers at all and $\mathbb{R}^d$ is the only free
  subspace. Assume we have proven the statement for dimensions larger
  than $d' (< d)$.

  The induction step is indirect.  We are going to show that, if the
  number of \(d'\) dimensional free subspaces is infinite, then for
  every positive \(\epsilon\) there exists a free subspace of higher
  dimension in \(\tilde Q^\epsilon\).  We will apply the inductive
  condition to \(\tilde Q^\epsilon\) ($\epsilon > 0$ sufficiently small) to derive a contradiction.

  For any given \(\delta>0\) there are only finitely many \(d'\)
  dimensional lattice subspaces, for which the lattice translates are
  \(\delta\)-separated.  By the indirect condition we have a free
  subspace, for which the lattice translates are \(\delta\)-dense in
  a higher dimensional subspace.  This higher dimensional subspace is
  therefore free in \(\tilde Q^\epsilon\) (as long as
  \(\epsilon<(1/7)\delta^2 \kappa_{\max}\)), but is not free in \(\tilde Q\)
(free subspaces can not contain each other by maximality).

  By the inductive assumption the number of higher (i.~e.~\(>d'\))
  dimensional free subspaces is finite.  For each \(\epsilon\) we
  create a vector \(\vec n^\epsilon\) such that the first coordinate
  is the number of \(d\) dimensional free subspaces for \(q\) in
  \(\tilde Q^\epsilon\), the second is the number of \(d-1\)
  dimensional free subspaces for \(q\) in \(\tilde Q^\epsilon\), and
  so on, the last coordinate is the number of \(d'+1\) dimensional
  free subspaces for \(q\) in \(\tilde Q^\epsilon\).  We consider the
  lexicographical ordering on these vectors, so the biggest is
  \((1,0,\dots,0)\), and \((0,2,3,0,1) > (0,2,2,23,11)\).  The set of
  possible vectors is not finite, but well ordered.

  We claim that \(\vec n^\epsilon\) does not increase as \(\epsilon\)
  decreases, and that \(\vec n^\epsilon\) is right continuous in
  \(\epsilon\).  For the first claim, observe that new free subspaces
  can only appear, if they were covered by higher dimensional free
  subspaces for higher \(\epsilon\) values.  So the first changing
  coordinate is decreasing.  For the second claim, observe that
  \(\tilde Q = \cap_{\epsilon>0} \tilde Q^\epsilon\), so if a free
  subspace is present for all small enough \(\epsilon>0\), then it is
  also present for \(\epsilon=0\).  Therefore \[\lim_{\epsilon\searrow
    0} \vec n^\epsilon = \min_{\epsilon> 0} \vec n^\epsilon = \vec
  n^0\] the first equality follows from monotonicity and
  well-orderedness, the second from right continuity.

  This is a contradiction with the previously proven statement: \(\vec
  n^\epsilon>\vec n^0\) for all \(\epsilon>0\). Indeed, by Lemma  \ref{sec:finstab} for any point \(q\in\tilde Q\) there exists an
\(\epsilon\) such that \(\vec n^\epsilon=\vec n^0\), and therefore the
free subspaces are the same.

\end{proof}

\begin{Lem}
\label{lem:finmax}
 There are finitely many maximal horizons.
 \end{Lem}
 \begin{proof}(see also Lemma 1 in \cite{D12} and Lemma A.2.2. of \cite{Sz94})
 For every $q \in Q$ pick a stability neighborhood using Lemma \ref{lem:stab}.
 Since $Q$ is compact, one can choose a finite cover of $Q$ by such neighborhood.
 This yields that there are only finitely many maximal dimensional
 free subspaces. It remains to prove that for such a free subspace $V$,
 there are only finitely many corresponding horizons. For this, project the scatterer
 configuration to $V^{\perp}$. Note that there is no higher dimensional
 free subspace than $V$, thus the complement of the images of the scatterers
 is the union of the bases of horizons with free subspace $V$. Since the
 complement of finitely many convex sets has finitely many connected components,
 the statement follows.
\end{proof}

\section{The proportionality lemma}
\label{sec:tail}

The next lemma states that any long enough free flight has a fixed
proportion of its time spent in a horizon without leaving it.  The
technical formulation is a little bit different, and formally we will
use the statement below, where instead of a horizon we use the
vicinity of a free subspace.

\begin{Lem}[Proportionality lemma]\label{lem:prop}
  For every \(\epsilon>0\) there exist \(T>0\) and \(c\in(0,1)\), such
  that for any \(x\in M\) if \(\infty > \tau(x) > T\) then there exist
  \(\tau(x)>s>t>0\) with \(s-t > c\tau(x)\) and a free subspace \(p+V
  \subset \tilde Q\) such that the configuration component of
  \(\Phi_u(x)\) is \(\epsilon\) close to \(p+V\) in \(\tilde Q\) for
  every \(s>u>t\).
\end{Lem}

\paragraph{Remark 1} As we will see in Section 5, this lemma is only used for handling the remainder term, i. e. the contribution of a countable union of smaller dimensional horizons in vicinities of  maximal horizons. This is why it is sufficient to ensure that only a positive proportion of a long free path is close to the free subspace of a horizon.

\paragraph{Remark 2}
The analogous Lemma in the planar case is much simpler:
for any long enough free flight (expect for its two extreme parts of
bounded length) is entirely spent in a horizon (see \cite{B92}).

\begin{proof}
  The proof is indirect.  We are going to suppose, that there exists
  an \(\epsilon>0\) such that for all \(T>0\) and \(c\in(0,1)\) there
  exists \(x\in M\) with \(\infty > \tau(x)>T\) such that for any free
  subspace \(p+V\subset \tilde Q\) and for any time segment
  \(\tau(x)>s>t>0\) if the configuration component of \(\Phi_u(x)\) is
  \(\epsilon\) close to \(p+V\) for \(s>u>t\), then
  \(\tau(x)>(s-t)/c\).

  \paragraph{Choice of constants}
  Choose \(T_n\to\infty\), and \(c_n\to 0\), and choose
  \((q_n,v_n)=x_n\in M\) according to the indirect assumption.  By
  compactness of \(M\) we have an accumulation point
  \(x_\infty=(q_\infty,v_\infty)\).  Apply lemma \ref{lem:stab} to get
  \(\xi\) as the stability fattening factor for \(q_\infty\).  We have
  an \(\epsilon\) from the indirect statement.  Choose \(\eta\), such
  that \(\frac32\eta<\epsilon\), and \(2^d\eta<\xi\).  For \(1 \leq k
  \leq d\) let us define
  \begin{equation}
    \label{eq:r_k}
    r_k = \frac{\ell_{k-1}}{\ell_d}\left( \frac{\eta}{2} \right)^{d-k}
    D^{d-k},
  \end{equation}
  where \(D^j\) is the \(j\)-dimensional volume of the
  \(j\)-dimensional unit ball, and \(D^0=1\).  Choose \(n\) such that
  \(|q_n - q_\infty|<\eta/2\) and
  \begin{eqnarray}
    \label{eq:large_T}
    T_n&>&\frac{1}{r_1}, \\
    \label{eq:small_c}
    c_n&<&2^{-d} \eta \min_{1\leq k \leq d} r_k.
  \end{eqnarray}

  \paragraph{Inductive assumptions}
  We are going to prove the following statements in an inductive
  fashion for \(1\leq k\leq d\).
  \begin{itemize}
  \item We have linearly independent lattice vectors
    \(\{l_i\}_{i=1}^k\), all from a free subspace for \(q_\infty\) in
    \(\tilde Q\).
  \item We have \(0<t_k<\tau(x_n)\), such that the parallelepiped
    \(q_n+\sum_{i=1}^k\lambda_i l_i\), \(\lambda_i \in [0,1]\) is
    contained in the \((2^k-1)\eta\) radius tubular neighborhood of
    the trajectory segment \(\Phi_{[0,t_k]}x_n\)
	(the $\rho$ tubular
    neighborhood of a line segment $[a,b]$ is the set of such points in
    $\mathbb{R}^d$ which are $\rho$-close to the line segment $[a,b]$,
    and whose orthogonal projection to the line defined by $a$ and $b$ lies
    between $a$ and $b$).
  \item Denote by \(v_n^{\perp k}\) the component of \(v_n\) which is
    orthogonal to \(\gen \{l_i\}_{i=1}^{k-1}\), this gives \(v_n\) for
    \(k=1\).  We require that:
  \begin{equation}
    \label{eq:t_k}
    t_k < \sum_{i=1}^k 2^{k-i} \frac{1}{\left| v_n^{\perp i} \right| r_i}
  \end{equation}
  \end{itemize}
  The last statement is purely technical.

  \paragraph{Start of induction}
  By condition (\ref{eq:large_T}) the tubular \(\eta/2\) neighborhood
  of the free flight trajectory of \(x_n\) has a bigger volume than
  \(\covol{\mathcal{L}}\), therefore it has a self intersection in
  \(Q=\tilde Q /\mathcal{L}\).  This means that, in this tubular
  neighborhood, there are two points \(q'\), and \(q'+l_1\) which are
  lattice translates with \(0\neq l_1 \in \mathcal{L}\).  Moreover
  \(|q'-q_n|<\eta/2\) and \(|q'+l_1- (q_n+t_1v_n)| <\eta/2\) and
  \(0<t_1<\tau(x_n)\).  Consequently in the fattened space \(\tilde
  Q^{\eta/2}\) the line segment \(q', q'+l_1\) is collision free and
  periodic, hence \(\tau^{\eta/2}(q',l_1)=\infty \).  Applying the
  stability lemma we conclude that \(l_1\) is part of a free subspace
  for \(q_\infty\) in \(\tilde Q\).  We also note that the line
  segment \(q_n,q_n+l_1\) is in the tubular \(\eta\) neighborhood of
  the trajectory segment \(\Phi_{[0,t_1]}x_n\).

  Note that we only used \(\tau(x_n)>1/r_1\) about the length of
  the free flight, so actually \(t_1<1/r_1\), which gives equation
  (\ref{eq:t_k}) for \(k=1\).

  \paragraph{Inductive step}
  Suppose we have all the inductive statements for \(k-1\).  For
  simplicity we denote the lattice subspace
  \(V=\gen\{l_i\}_{i=1}^{k-1}\), and its orthocomplement \(V^\perp\).
  Consider the orthogonal projection of the free flight
  \(\Phi_{[0;\tau(x_n)]}^\perp(x_n)\).  Since
  \(|q_n-q_\infty|<\eta/2\) and \(\epsilon>3\eta/2\) the projection of
  the free flight lies in at least \(\eta\) length (and equivalently
  for at least \(\eta/|v_n^{\perp k}|\) time) in the \(\epsilon\)
  neighborhood of \(q^\perp_\infty\), meaning that the non projected
  free flight spends the same \(\eta/|v_n^{\perp k}|\) time in the
  \(\epsilon\) neighborhood of the free subspace containing \(V\).
  By the indirect condition, the complete length of the projection
  \(|v_n^{\perp k}|\tau(x_n)\) is at least \(\eta/c_n>1/r_k\),
  therefore
  \begin{equation}
    \label{eq:tau_c}
    \tau(x_n)> \frac{\eta}{|v_n^{\perp k}|c_n}
  \end{equation}

  \begin{figure}
    \centering
    \begin{tikzpicture}[scale=.8]
      \draw[thick] (0,0) node[right] {\(q_n\)} --
      (-6,4.5) node[above] {\(q_n+(t+t')v_n\)};
      \draw (1,1) -- (1,5) -- (5,2) -- (5,-2) -- cycle;
      \draw (1.5,4) node {\(V^\perp\)};
      \draw[->] (1,4) -- +(-1,0) node[anchor = south] {\(V\)};
      \draw[thick] (4,0) node [right] {\(q_n^\perp\)} -- (2,3);
      \draw (3.5,-.5) -- (1.5,2.5) node [left, fill = white]
      {\(q'^\perp + l_k^\perp\)} -- (2.5,3.5) -- node[below,sloped]
      {tubular nbh} (4.5,.5)
      node[right, fill = white] {\(q'^\perp\)} -- cycle;
      \draw (-.5,-.5) -- (.5,.5) node[right] {\(q'\)};
      \draw (-4.5,2.5) node[anchor = north] {\(q'+l_k-v\)} -- +(1,1);
      \filldraw (-6,4.5) circle (2pt);
      \filldraw (4,0) circle (2pt);
      \filldraw (4.5,.5) circle (2pt);
      \filldraw (1.5,2.5) circle (2pt);
      \filldraw (0,0) circle (2pt);
      \filldraw (.5,.5) circle (2pt);
      \filldraw (-4.5,2.5) circle (2pt);
      \filldraw (-7.5,2.5) circle (2pt) node[below] {\(q'+l_k\)};
      \filldraw (0,-.7) circle (2pt) node[below] {\(q_\infty\)};
      \filldraw (-4,3) circle (2pt) node[left] {\(q_n+tv_n\)};
      \draw[dotted] (-3.5,3.5) -- node {projection} (2.5,3.5);
    \end{tikzpicture}
    \caption{Constellation of vectors in the inductive step in the
      proof of lemma \ref{lem:prop}}
    \label{fig:const}
  \end{figure}
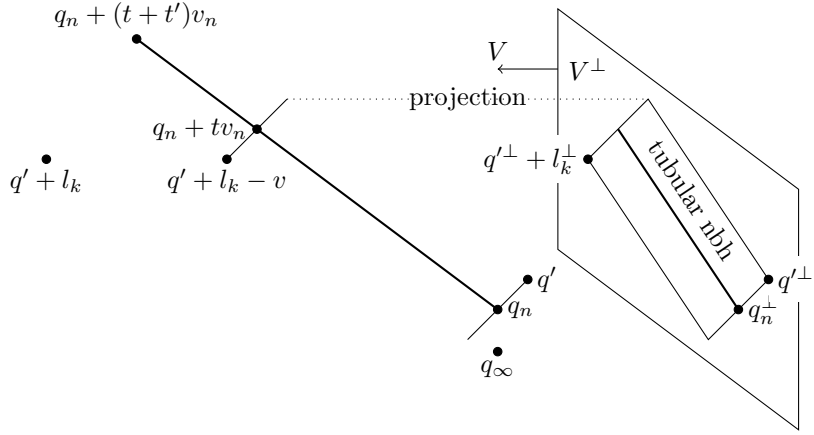
  By the definition of \(r_k\) we have that the
  ((\(d-k+1\))-dimensional) volume of the tubular \(\eta/2\)
  neighborhood of the projected free flight trajectory, is bigger
  than \(\ell_d/\ell_{k-1}\), so in particular bigger than
  \(\covol{\mathcal{L}}/\covol[V]{\mathcal{L}\cap V}\), which is by
  Claim \ref{lem:prod_covol} the covolume of the projected lattice
  \(\mathcal{L}_V^\perp\).  Therefore this neighborhood contains a
  pair of points \(q'^\perp\), and \(q'^\perp + l^\perp_k\), with
  \(0\neq l^\perp_k\in\mathcal{L}_V^\perp\).  The latter means that
  there is \(l_k\in\mathcal{L}\setminus V\), such that \(l^\perp_k\)
  is its projection.  We can choose \(q'\) such that
  \(|q'-q_n|<\eta/2\) and
  \begin{equation}
    \label{eq:ind_para}
    \left|q'+l_k -(q_n+t v_n) - v\right|<\eta/2,
  \end{equation}
  for some \(\tau(x_n)>t>0\), and some \(v\in V\).  We can suppose,
  that \(v\) is in the parallelepiped \(\sum_{i=1}^{k-1} \lambda_i
  l_i\), since the lattice component can be added to \(l_k\), it does
  not change the property, that \(l_k\in \mathcal{L}\setminus V\).
  The inductive condition gives that \(v\) is in the tubular
  \((2^{k-1}-1)\eta\) neighborhood of the trajectory segment
  \(\Phi_{[0,t_{k-1}]}x_n\), we have from equation \ref{eq:ind_para}
  that
  \begin{equation}
    \label{eq:ind_abs}
    \left|q'+l_k -(q_n+(t + t') v_n) \right|<\left( 2^{k-1}-\frac12 \right)\eta,
  \end{equation}
  where \(0<t'<t_{k-1}\).  The positivity of \(t'\) comes from the
  sign of \(v\) in equation \ref{eq:ind_para} and the fact, that all
  \(l_i\) has positive scalar product with \(v_n\) by construction.
  It follows, that the line segment \(q_n,q_n+l_k\) is in the tubular
  \(2^{k-1}\eta\) neighborhood of \(\Phi_{[0,t+t']}x_n\), and
  therefore the parallelepiped \(q_n+\sum_{i=1}^{k}\lambda_i l_i\) is
  in the \(\left( 2^k-1 \right)\eta\) neighborhood of
  \(\Phi_{[0,t+t'+t_{k-1}]}x_n\).

  We declare \(t_k= t+t'+t_{k-1}\), and note that in the construction
  of \(t\) we have only used \(\tau(x_n)>1/r_k\left| v_n^{\perp k}
  \right|\) about the length of the free flight, so actually
  \(t<1/r_k\left| v_n^{\perp k} \right|\).  Using \(t'<t_{k-1}\), and
  equation \ref{eq:t_k} from the inductive condition for \(k-1\) we
  get \[t_k<\frac{1}{r_k\left| v_n^{\perp k} \right|} +
  2\sum_{i=1}^{k-1} 2^{k-1-i} \frac{1}{\left| v_n^{\perp i} \right|
    r_i} = \sum_{i=1}^k 2^{k-i} \frac{1}{\left| v_n^{\perp i} \right|
    r_i},\] which is equation \ref{eq:t_k} for \(k\).  To show
  \(t_k<\tau(x_n)\), observe, that \(\left| v_n^{\perp i} \right|\) is
  decreasing with \(i\), hence \[t_k<\sum_{i=1}^k 2^{k-i}
  \frac{1}{\left| v_n^{\perp i} \right| r_i} < \sum_{i=1}^k 2^{k-i}
  \frac{1}{\left| v_n^{\perp k} \right| \min r_i} < 2^k
  \frac{1}{\left| v_n^{\perp k} \right| \min r_i} \leq
  \frac{\eta}{\left| v_n^{\perp k} \right| c_n}.\] The last inequality
  follows from equation \ref{eq:small_c}.  The last expression in the
  row, and hence \(t_k\) is smaller than \(\tau(x_n)\) by equation
  \ref{eq:tau_c}.

  In the fattened space \(\tilde Q^{\left(2^k-1\right)\eta}\) we
  have a \(k\) dimensional lattice parallelepiped, and (by
  \(\mathcal{L}\) periodicity) the generated lattice subspace free of
  scatterers.  By the choice of \(\eta\) we can apply the stability
  lemma to conclude, that \(\{l_i\}_{i=1}^k\) are from a free
  subspace for \(q_\infty\) in \(\tilde Q\).

  \paragraph{Contradiction}
  The last (\(k=d\)) claim in the above induction states the existence
  of a \(d\) dimensional free subspace, which means that there are no
  scatterers.  Even in that case the indirect condition states that
  the trajectory leaves this free subspace, which is the whole
  configuration space.
\end{proof}

\section{Proof of Theorem \ref{tetel1}}
\label{sec:proof}
Here, we prove the generalization of Dettmann's first conjecture
(i.e. Theorem \ref{tetel1}).

\subsection{Lower estimate}
\label{sec:lower}

First, we prove the lower estimate, namely
\begin{equation}
\label{eq:lower_est}
{\lim \sup}_{t \rightarrow \infty}
 \sum_{H \in \mathbb H} F_H(t) /F(t) \leq 1.
\end{equation}
Since $ \cup_H \left\{ (q,v) \in M
\mid q+sv \in H, \enskip \forall s \in [0,t] \right\} \subset
\{ (q,v) \in M \mid \tau(q,v)>t \}$, (\ref{eq:F_H}) implies that
(\ref{eq:lower_est}) follows, whenever
\begin{equation*}
\mu \left( \left\{ (q,v) \in M
\mid \exists H_1 \neq H_2 \in \mathbb H, \enskip \forall s \in [0,t],
\enskip q+sv \in H_1 \cap H_2 \right\} \right)
= o \left( t^{d_H-d} \right)
\end{equation*}
is established.
Since there are finitely many maximal horizons, it suffices to prove that
for every pair $(H_1, H_2) \in \mathbb H ^2$,
\[ F_{H_1, H_2}(t) = \mu \left( \left\{ (q,v) \in M
\mid q+sv \in H_1 \cap H_2, \enskip \forall s \in [0,t]
\right\} \right) = o \left( t^{d_H-d} \right). \]
Now assume that for fix $(H_1, H_2)$ and for every $n>1$, one can find
$x_n \in M$ such that the trajectory segment $\Phi_{[0,n]} x_n$ lies
entirely in $H_1 \cap H_2$ (if not, then obviously
$F_{H_1, H_2}(t) =0$ for $t$ large enough). Since maximal horizons are
closed, there is an accumulation point $x_{\infty} = (q_{\infty}, v_{\infty} )$
with $\Phi_{[0,\infty]} x_{\infty} \in H_1 \cap H_2$. Thus the
set $V_{H_1, H_2} = V_{H_1} \cap V_{H_2}$ is a non-empty subspace of
$V_{H_1}$. Obviously it is strictly smaller than $V_{H_1}$,
otherwise $H_1$ and $H_2$ would coincide. Now project the scatterer
configuration to $V_{H_1, H_2}^{\perp}$. In this projection, the intersection
of the images of $\tilde H_1$ and $ \tilde H_2$ does not contain any subspace
(indeed, if it contained a line, that could be added to $V_{H_1, H_2}$).
Then the same argument
used to prove (\ref{eq:F_H}) provides
$F_{H_1, H_2}(t) = O \left( t^{d_H-1-d} \right)$.

\subsection{Upper estimate}
\label{sec:upper}

The estimate will work as an induction by dimension.  If \(d=1\) the
claim is trivial, the \(d=2\) case was proved in \cite{SzV07}.\\
The idea of the present proof is briefly the following. The measure of points
for which the trajectory up to time $t$ is spent in a horizon of dimension
$d'$ is of order $t^{d'-d}$. In order to prove the upper bound, one
needs to overcome two difficulties. First, there are trajectories which
travel from one horizon to another - this problem is solved by the Proportional
lemma. The second problem is that although there are finitely many maximal horizons,
but there are infinitely many lower dimensional ``attached'' horizons, thus the above naive
estimation cannot be summed up. To solve this problem, we slightly extend
the maximal horizons in the estimation - this way, they swallow all, but finitely many
attached horizons, while their leading constant ($C_H$) do not change a lot.\\
Formally, in the general $d$ dimensional case, we prove the following statement.
For every $\delta >0$ there exists a $T<\infty$
such that for every $t > T$,
\begin{equation}
\label{eq:upper_est}
F(t) \leq (1+ \delta) \sum_{H \in \mathbb H} C_{H} t^{d_{\max}-d},
\end{equation}
where $d_{\max}$ is the dimension of the maximal horizons.
To prove this, let us introduce the fattened version of the
maximal horizons. Since $\cap_{\varepsilon >0} \tilde{Q}^{\varepsilon}
= \tilde{Q}$,
for $\varepsilon$ small enough, the maximal horizons
of the fattened configuration space $\tilde{Q}^{\varepsilon}$ are in one to
one correspondence with those of $\tilde{Q}$, and are slightly thicker then
those. Thus one can choose $\varepsilon >0$ such that
\[ \sum_{H \in \mathbb H} C_{H^{3 \varepsilon}} < (1+ \delta/4)
\sum_{H \in \mathbb H} C_{H},\]
where $H^{3 \varepsilon}$ is the fattened version
of the horizon $H$ - which can also be written as $B_H^{3 \varepsilon} \times V_H$ -  and
$C_{H^{3 \varepsilon}}$ is the corresponding constant defined in (\ref{eq:F_H}).
Note that the ${3 \varepsilon}$
neighborhood of $H$ ($B_H$, resp.) is a proper subset of
$H^{3 \varepsilon}$ ($B_H^{3 \varepsilon}$, resp.).
Fix this $\varepsilon$ for the rest of the proof.

\paragraph{Estimator environments}

Now, for any fixed $\varepsilon > 0$, we construct a finite net of environments, called estimator environments, which will be used by the estimate. In fact, this finiteness will have an essential role in our arguments so despite of its simplicity we formulate the statement in a lemma.
\begin{Lem}
Given $\varepsilon > 0$, one can find a finite set of points $q_1, \dots, q_{\overline i}$ with $\mathcal V(q_i) = \{ V_j(q_i): 1 \le j \le j(q_i)\}$ such that for arbitrary $q \in Q$ and any free subspace $V \in \mathcal V(q)$, there are some $q_i$
and $j:  1 \le j \le j(q_i) $ such that $q$ is in the $\varepsilon$ neighborhood of $q_i$ and
$V \subset V_j(q_i)$.

Consequently, the $2 \varepsilon$ neighborhood of
$q_i+V_{i,j}$ contains the $\varepsilon$ neighborhood of $q + V$.
\end{Lem}

\begin{proof}
Using Lemma \ref{lem:stab}, for every point $q \in Q$ pick a stability neighborhood $U(q)$ of radius $\xi(q) < \varepsilon $. By the compactness of $Q$, fix a finite subcover $\cup_{i=1}^{\overline i} U(q_i)$ of $Q$ from these environments and remember that by Lemma \ref{lem:fin} each $\mathcal V(q_i) = \{V_j(q_i): 1 \le j \le j(q_i)\}$ is finite. Then by the definition of stability neighborhoods, we have that
for arbitrary $q \in Q$ and  any free subspace $V \in \mathcal V(q)$, there are some $q_i$
and $j:  1 \le j \le j(q_i) $ such that $q$ is in the $\varepsilon$ neighborhood of $q_i$ and
$V \subset V_j(q_i)$.
\end{proof}

\noindent{\bf Remark}
Those $q_i+V_{i,j}$'s with $\dim V_{i,j}=d_{\max}$ are necessarily subsets
of some maximal horizons. Since $H^{2 \varepsilon}$ contains the
$2 \varepsilon$ neighborhood of $H$, the $2 \varepsilon$ neighborhoods
of these $q_i+V_{i,j}$'s are covered by the $H^{2 \varepsilon}$'s.
Thus we call the sets $H^{2 \varepsilon}$ for
$H \in \mathbb{H}$, and the $2 \varepsilon$ neighborhoods
of the remaining $q_i+V_{i,j}$'s ($i \in I, j \in J_i$)
{\it estimator environments}. Remind that $\dim V_{i,j} < d_{\max}$ for all
$i \in I, j \in J_i$, and that the $\varepsilon$ neighborhood of
any affine free subspace is covered by some estimator environment
- thus the Proportionality lemma asserts that the $c$ portion of a long
enough free flight is spent in an estimator environment.

\paragraph{Proof of (\ref{eq:upper_est})}
First, with the already fixed $\varepsilon$, use the Proportionality lemma
to obtain some $c$ and $T$. From now on, we always assume $t>T$.
For the estimation of $\mu(\tau >t)$, we distinguish three cases.\\

{\bf Case 1} {\bf Such points $x=(q,v) \in M$ with $\tau(x)>t$, for which the
time interval $[s_1,s_2]$ with $0 < s_1<s_1+c\tau(x)<s_2<\tau(x)$ guaranteed by the Proportionality lemma
is spent in the $2 \varepsilon$ neighborhood of $q_i + V_{i,j}$ for some
$i \in I, j \in J_i$.}\\
Since there is a line segment of length at least $ct/2$ spent in the neighborhood of $q_i + V_{i,j}$,
the angle of $v$ and $ V_{i,j}$ is necessarily smaller than $2/(ct)$. As it was also used by the
proof of (\ref{eq:F_H}), the $d-1$ dimensional Lebesgue measure on ${S}^{d-1}$ of such
velocity vectors $v$ is asymptotically
\[ \left( \frac2{ct} \right)^{\dim V_{i,j} -d}. \]
Since $\dim  V_{i,j} < d_{\max}$ and there are finitely many estimator environments, for $t$
large enough the $\mu$-measure of points of Case 1 are smaller than
\[ \delta/4 \sum_{H \in \mathbb H} C_{H} t^{d_{\max} -d}.\]

{\bf Case 2} {\bf (Main term) Such points $x \in M$ with $\tau(x)>t$, where the configuration component
of $\Phi_{[0,t]}x$ is a subset of $H^{3 \varepsilon}$ for some $H \in \mathbb{H}$.}\\
The same argument used to prove (\ref{eq:F_H}) implies that the
$\mu$-measure of such points is asymptotically not larger than
\[ \sum_{H \in \mathbb H} C_{H^{3 \varepsilon}} t^{d_{\max} -d}, \]
thus for $t$ large enough, is smaller than
\[ (1+ \delta/2) \sum_{H \in \mathbb H} C_{H} t^{d_{\max} -d}.\]

{\bf Case 3} {\bf Such points $x \in M$ with $\tau(x)>t$ not treated in Case 2, for which the
time interval $[s_1,s_2]$ with $0<s_1<s_1+c\tau(x)<s_2<\tau(x)$ guaranteed by the Proportionality lemma
is spent in $H^{2 \varepsilon}$ for some $H \in \mathbb{H}$.} It is worth noting that one difficulty of this case comes from the fact that it covers an infinite number of lower dimensional "attached" horizons. \\
Note that $\Pi_{Q} \Phi_{[0,\tau(x)]}x$ for such an $x$ has a part of
length at least $c\tau(x) /2$
in the region $H^{3 \varepsilon} \setminus H^{2 \varepsilon}$ and also crosses this region
in the sense that intersects with both $H^{2 \varepsilon}$ and the complement
of $H^{3 \varepsilon}$.
Thus there are some $s_3,s_5$ with $0<s_3<s_3+c \tau(x)/2 <s_5<\tau(x)$ such that
$\Pi_Q \Phi_{s_3} x$ is in $\partial B_{H^{2 \varepsilon}} \times
V_H$ and $\Pi_Q \Phi_{s_5} x$ is in $\partial B_{H^{3 \varepsilon}}
\times V_H$ (or $\Pi_Q \Phi_{s_3} x$ is in $\partial B_{H^{3 \varepsilon}} \times
V_H$ and $\Pi_Q \Phi_{s_5} x$ is in $\partial B_{H^{2 \varepsilon}}
\times V_H$, which case can be treated analogously).
As a starting idea, one can think about this trajectory segment as a long free flight
in a $d_{\max}$ dimensional billiard, which guarantees that the Lebesgue measure of points
of Case 3 are not large. More precisely, write
\begin{equation*}
\Phi_{s_3} x = (q^{\perp} + q^{\parallel}, v^{\perp} + v^{\parallel}),
\end{equation*}
where $q^{\perp}$ and $v^{\perp}$ are in $V_H^{\perp}$, while
$q^{\parallel}$ and $v^{\parallel}$ are in $V_H$. Note that $q \in Q$ by definition, but
the components $q^{\perp}, q^{\parallel}$ are in $\mathbb{R}^d$.
The projection of the
trajectory segment $\Pi_{\tilde Q} \Phi_{[s_3, s_5]} x$ to $V_H^{\perp}$, prescribed by
$q^{\perp}$ and $v^{\perp}$, is going to be used to construct the billiard
table of dimension $d_{\max}$, while $q^{\parallel}$ and $v^{\parallel}$
are going to define the trajectory in this lower dimensional billiard table.
There is a point $z \in \partial B_H$
such that in the intersection point of $z +V_H$ and $\partial {Q}$
the $d$ dimensional sphere of radius $\kappa_{\max}^{-1}$ touching the appropriate scatterer
from inside has a center,
the
projection of which to $V_H^{\perp}$
is collinear with $z$ and $q^{\perp}$.
(See Figure \ref{fig:V_H0}.) Let us
denote this sphere by $S$.
Now, consider the affine subspace $q^{\perp} + V_H$.
By definition, there exists a point $q^{\perp} + p$
in this affine subspace such that the $d$ dimensional ball
of radius $2 \varepsilon$ and center $q^{\perp} + p$ is contained
completely in $S$ and hence in a scatterer.\\
Now let us define a $d_{\max}$ dimensional billiard configuration space:
the periodicity is $\mathcal{L} \cap V_H$, there is one spherical scatterer
of radius $\varepsilon$ and the center of this spherical scatterer is $p$
(when $\mathbb{R}^{d_{\max}}$ is identified with $V_H$). Denote its
configuration space by $\tilde Q_{d_{\max}}$.
Note that the intersection
of $\tilde Q$ and $q^{\perp} + V_H$ is contained in $\tilde Q_{d_{\max}}$
(again, with $\mathbb{R}^{d_{\max}}$ being identified with $V_H$).
Further, we claim that with the notation
\[ s_4 = s_5 \wedge \min \{ s>s_3:
\dist (q^{\perp}, q^{\perp}+ (s - s_3) v^{\perp}) >
{ \varepsilon}\}, \]
for every $s_3 <s <s_4$, the intersection of $\tilde Q$ and
$q^{\perp} + (s-s_3)v^{\perp} + V_H$ is also contained in $\tilde Q_{d_{\max}}$.
Indeed, since $\dist (q^{\perp}, q^{\perp}+(s-s_3)v^{\perp}) <
{ \varepsilon}$, the $d$ dimensional ball of radius $\varepsilon$ and center
$q^{\perp}+(s-s_3)v^{\perp} + p$ is contained in the ball of radius $2 \varepsilon$ and center
$q^{\perp} + p$. The latter statement is in general not true for $s=s_5$,
since $q^{\perp}+(s_5-s_3)v^{\perp}$ can be outside of the projection of $S$ (see
Figure \ref{fig:V_H0}), that is why we needed to introduce $s_4$.\\
Now, we can easily map a long
free flight in this $d_{\max}$ dimensional billiard to our trajectory segment
$\Phi_{[s_3, s_5]}(x)$. Namely, let us choose the free flight of the phase point
$(q^{\parallel}, v^{\parallel})$ in $Q_{d_{\max}}$. Due to the construction, this
free flight is longer than $(s_4 - s_3)/2$.
We claim that this is longer
than a universal constant (in the sense that does not depend on $x$ but may
depend on $\varepsilon$ and also on $H$ since there are finitely many of them) times $t$, i.e.
\begin{Lem}
\label{lem:case3}
There is a constant $c'(\varepsilon)$, such that
$s_3<s_3+2 c'(\varepsilon) \tau(x) < s_4 \leq s_5$.
\end{Lem}

\begin{figure}
\centering
\begin{tikzpicture}
\draw  (0cm,10cm) arc (225: 270: 10cm);
\draw[dashed]  (1cm,11cm) arc (225: 265: 10cm);
\draw[dashed]  (1.2cm,11.7cm) arc (225: 270: 9.5cm);
\draw (3.5,8) node {$z$}
      (4, 9.2) node {$q^{\perp}$}
            (4.2,10.1) node {$q^{\perp} + (s_4-s_3) v^{\perp}$}
	          (8,7) node {$\partial B_H$}
		        (8,8) node {$\partial B_{H^{2 \varepsilon}}$}
			      (8,8.7) node {$\partial B_{H^{3 \varepsilon}}$};
			      \draw (1.6cm,11.7cm) arc (160:330:3.08cm);
			      \draw[color=red] (3.8,9) -- (1.2, 11.7);
			      \draw (1.2, 11.9) node {$q^{\perp} + (s_5-s_3) v^{\perp}$};
			      \filldraw[color=red] (1.2, 11.7) circle (1pt);
			      \filldraw[color=red] (3.8, 9) circle (1pt);
			      \filldraw[color=red] (3, 9.8) circle (1pt);
			      \filldraw (3.3, 7.8) circle (1pt);

			      \end{tikzpicture}
			      \caption{Construction of the $d_{\max}$ dimensional billiard table - figure in $V_H^{\perp}$
			      }
			      \label{fig:V_H0}
			      \end{figure}
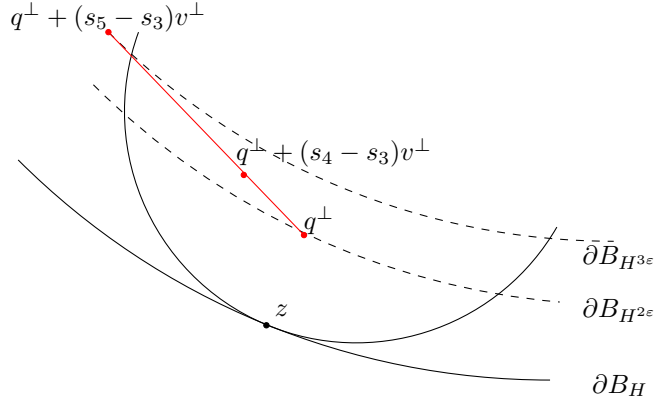

\begin{proof}
It is enough to prove that there exists some $c''(\varepsilon)$
such that $s_3+c''(\varepsilon) (s_5 - s_3) < s_4$. Since
$|(s_4-s_3) v^{\perp}| > \varepsilon$, it is enough to give an upper bound
for $|(s_5-s_3) v^{\perp}|$. Thus we need that the function
\begin{equation}
\label{s_4korlat}
\Delta(y,z) = \max \{ r | \exists l \in \mathcal{L}: \overline{y, z+l} \subset B_{H^{3 \varepsilon}}
\text{and} \dist(y, z+l)=r\}
\end{equation}
on $B_{H^{3 \varepsilon}} \times B_{H^{3 \varepsilon}}$ is bounded (then
$\varepsilon$ divided by this bound is an appropriate choice for $c''(\varepsilon)$).\\
In order to see that (\ref{s_4korlat}) is bounded, first we prove that the set $\{ \Delta(y,z) \geq n \}$
is closed for any integer $n$. Choose any convergent sequence $(y_i, z_i) \rightarrow
(y_{\infty}, z_{\infty})$ from the above set. There are
corresponding $l_i \in \mathcal{L}$ vectors by the definition of $\Delta(y,z)$. Then the set
$\{ l_i: i\geq 1\}$ cannot be infinite, since if it was, then one could choose a convergent subsequence
of $l_i/|l_i|$ and the line with this direction containing $y_{\infty}$
would be a subset of $B_{H^{3 \varepsilon}}$ which is a contradiction. Thus the set
$\{ l_i: i\geq 1\}$ is finite. Hence one can choose a subset $(y_{i_k}, z_{i_k}) \rightarrow
(y_{\infty}, z_{\infty})$ with $l_{i_k} =l$, yielding $\Delta(y_{\infty}, z_{\infty}) \geq n$. Whence
$\{ \Delta(y,z) \geq n \}$ is closed. Now assume by contradiction that $\Delta(y,z)$ is not bounded,
thus the sets $\{ \Delta(y,z) \geq n \}$ for $n \geq 1$ are closed subsets of each other. Thus
there is a pair $(y,z)$ such that $\Delta(y,z) = \infty$. Just like before, one can easily deduce
the existence of an infinite line in $B_{H^{3 \varepsilon}}$ through $y$ which is a contradiction.
Thus we have proved that (\ref{s_4korlat}) is bounded and thus verified the existence of an
appropriate $c'(\varepsilon)$.
\end{proof}
Now, we finish the proof of Case 3. Since at least $ct$ time of the free flight
is spent in $H^{2 \varepsilon}$, the angle of $v$ and $V_H$ is smaller then
$C_1 /t$ with some $C_1$. Thus the points $x=(q,v)$ of Case 3 are elements of the set
$\mathbb R^d / \mathcal L \times V(t)$, where
\[ V(t) = \{ v \in S ^{d-1}: \angle( v, V_H) < C_1/ t \}.\]
As before, $\lambda_{d-1} (V(t)) < C_2 t^{d_{\max} -d}$. Every point $(q,v) \in
\mathbb R^d / \mathcal L \times V(t)$ can uniquely be written in the form
\[ (q,v) = (q_0^{\perp}+ q_0^{\parallel}, v^{\perp}+ v^{\parallel}) \]
with $q_0^{\perp}, v^{\perp} \in V_H^{\perp}$,
$q_0^{\parallel}, v^{\parallel} \in V_H$. The conditional measure of $\lambda_{d} \times
\lambda_{d-1}$ on $\mathbb R^d / \mathcal L \times V(t)$ to such points where
$q_0^{\perp}, v^{\perp}$ are fixed, is also Lebesgue on the possible set of
pairs $(q_0^{\parallel}, v^{\parallel})$. Note that since $|v^{\perp}|$ is small,
the set of possible $v^{\parallel}$'s is a $d_{\max}$ dimensional sphere of radius
close to one. But the set of possible $q_0^{\parallel}$'s depends on $q_0^{\perp}$,
since $V_H^{\perp}$ is not necessarily generated by lattice vectors. Thus write
\[ \mathfrak q(q) = \{ \bar{q} \in \mathbb R^d / \mathcal L: \bar{q}_0^{\perp}
= {q}_0^{\perp} \}. \]
One can easily prove that there exists some $\eta >0$ such that
\[ \lambda_d \{ q: \lambda_{ d_{\max}} (\mathfrak q (q)) < \eta \} <
\frac{C_H \delta}{8 c_{\mu} C_2}. \]
Thus
\[ \mu ( \{ q: \lambda_{ d_{\max}} ( \mathfrak q(q)) < \eta \}
\times V(t)) < \frac{\delta}{8} C_H t^{d_{\max}-d}.\]
Now we can assume that
\begin{equation}
\label{case3_1}
\lambda_{ d_{\max}} (\mathfrak q(q)) > \eta.
\end{equation}
Since the $c$ portion of the line segment
$\Pi_{V_H^{\perp}} \Pi_Q \Phi_{[0, \tau (x)]}x$ is spent in
$H^{2 \varepsilon}$, once
$q_0^{\perp}, v^{\perp}$ are fixed, the number of possible $q^{\perp}$'s
(that is, the projection of $\Pi_Q \Phi_{s_3}x$ to $V_H^{\perp}$) is bounded.
This, the inductive hypothesis (used on the billiard table
$\tilde Q_{d_{\max}}$), Lemma \ref{lem:case3} and (\ref{case3_1}) imply that
once $q_0^{\perp}, v^{\perp}$ are fixed,
the $ \lambda_{d_{\max}} \times \lambda_{d_{\max}-1}$ measure of such
coordinates $(q_0^{\parallel}, v^{\parallel})$ with which the free flight is
longer than $t$ is bounded by some universal constant times $t^{-1}$. Consequently, for $t$
large enough, the $\mu$ measure of points in Case 3 are smaller than
\[ \frac{\delta}{4} \sum_{H \in \mathbb H} C_H t^{d_{\max}-d}. \]

\section{Proof of Theorem \ref{tetel2}}
\label{sec:inc}

\subsection{Lorentz process with small scatterers}
\label{sec:BG}
First, we recall the following result of Bourgain, Golse and Wennberg (see
\cite{BGW98} and \cite{GW00}).\\
Consider a billiard table with periodicity $\mathbb Z ^{D}$ ($
D \geq 2$) and one spherical
scatterer of radius $r <1/2$. Define $\mu_{\mathbb Z,r}$ and
$\tau_{\mathbb Z,r}$ for this billiard table as before. Then there exist
$c'(D)$ and $C'(D)$ such that
\begin{equation}
\label{BGWbecslesZ}
\frac{c'(D)}{tr^{D-1}} \leq \mu_{\mathbb Z,r}
(\tau_{\mathbb Z,r} >t) \leq \frac{C'(D)}{tr^{D-1}}
\end{equation}
is true whenever
\begin{equation}
\label{BGWfeltetelZ}
t>r^{1-D}.
\end{equation}
In the case $t \approx r^{D-1}$, the so-called Boltzmann-Grad limit, much more
is known than (\ref{BGWbecslesZ}), see \cite{MS10}, or Remark 8.3.\\
In order to prove Theorem \ref{tetel2}, we need a slightly extended version
of the above estimation.\\
Let $\mathcal{L}'$ be any $D$-dimensional lattice and let
$q_1, \dots q_{n'} \in \mathbb{R} ^D / \mathcal{L}'$.
Consider the billiard table
with periodicity $\mathcal{L}'$
and finitely many disjoint spherical scatterers of radius $r$ centered at
$q_1, \dots q_{n'}$.
Let $Q'$, $M'$, $\mu'$ and $\tau'$ be defined accordingly.

\begin{Lem}
There exist
$c'(\mathcal L')$ and $C'(\mathcal L')$ such that
\begin{equation}
\label{BGWbecsles}
\frac{c'(\mathcal L')}{tr^{D-1}} \leq
\mu' (\tau' >t) \leq \frac{C'(\mathcal L')}{tr^{D-1}}
\end{equation}
is true whenever
\begin{equation}
\label{BGWfeltetel}
t>r^{1-D}.
\end{equation}

\end{Lem}

\noindent{\bf Remark} Obviously, Lemma \ref{lemma:inc} also implies
that for any fixed $\eta>0$, (\ref{BGWbecsles}) is true if $t>
\eta r^{1-D}$, with some $c'(\mathcal L')$ and $C'(\mathcal L')$
depending also on $\eta$. Thus, whenever we refer to
(\ref{BGWfeltetel}), it may be true only with some $\eta$,
but in order to make the exposition simpler, we do not
keep track of the $\eta$'s.

\begin{proof}
First, we prove the upper estimate.
Pick a basis $\{a_i\}_{i=1}^{D}$ of the lattice $\mathcal L'$ and denote by
$A$ the matrix whose $i$-th column is $a_i$. Also write $\sigma_i$ for the
$i$-th smallest singular value of $A^{-1}$.
Further, identify $\mathbb{R}^{D}
/ \mathbb{Z}^{D}$ with the unit cube and $\mathbb{R}^{D}
/ \mathcal L'$ with the parallelepiped $(a_i)_{i=1}^{D}$. Without
loss of generality, we may assume that one of the spherical
scatterers is centered
at the origin (i.e. $q_1=0$).\\
Now assume that for some $x'=(q',v') \in M'$, $\tau' (x') > t$. Then for the point
\[\phi(x') := x_{\mathbb Z} = ( A^{-1} q', \frac{A^{-1} v'}{\|A^{-1} v'\|}),\]
we have
\[ \tau_{\mathbb Z, r \sigma_1} (x_{\mathbb Z} ) > t \sigma_1.\]
Indeed, the image under $A^{-1}$
of the sphere of radius $r$ centered at the origin contains the sphere of
radius $r \sigma_1$ (the images of the possible other scatterers are simply omitted).
The Lebesgue measure on $Q'$ is transformed by $\phi$ to $\det (A^{-1})$ times the
Lebesgue measure in $\mathbb{R}^{D} / \mathbb{Z}^{D}$ minus an ellipse centered at
the origin, which is dominated by the Lebesgue measure on
$\mathbb{R}^{D} / \mathbb{Z}^{D} \setminus B(0, r \sigma_1)$. The image of the Lebesgue
measure on $S^{D-1}$ by $\phi$ is $\frac1{\|A^{-1} v'\|} dv'$.\\
Thus, using (\ref{BGWbecslesZ}), one can prove the second part of (\ref{BGWbecsles})
with
\[C'(\mathcal L') = \det (A^{-1}) \sigma_n  \sigma_1^{-D-1} C'(D)\]
at least, for
$t> \sigma_1^{-D} r^{D-1}$, but consequently for $t> r^{D-1}$ too, possibly with
a different $C'(\mathcal L')$.\\
Now, we prove the lower estimate. Observe that it is enough to prove
the statement for the special case $\mathcal L' = \mathbb Z^{D}$. Indeed, once
$c(\mathbb Z^{D})$ is found, one can prove the existence of
$c(\mathcal L')$ for any $\mathcal L'$
the same way as in the upper estimation.\\
Thus the statement we are going to prove is indeed a slight modification
of the first part of (\ref{BGWbecslesZ}): the difference is that we have
$n'$ spherical scatterers of radius $r$ centered at arbitrary points
$q_1, \dots q_{n'}$, instead of just one scatterer. We claim that an
obvious modification of the proof of Golse and Wennberg applies here.
Indeed, if $q$ is an integer vector with $g.c.d.(q)=1$ and one
projects the scatterer configuration to the line with direction $q$,
then observes a gap of length at least
$(1/|q| - 2n'r)/n'$
among the images of the scatterers,
assuming of course that $r<(2n'|q|)^{-1}$. Hence
there is a principal horizon perpendicular to $q$
(or ``sandwich layer'') whose middle third
has width
\[ a_{q,r} = \frac13 \left( \frac1{|q|} - 2n'r \right) \frac1{n'}. \]
Considering only those $q$'s for which $|q|<q_{\max} = (4n'r)^{-1}$,
the density of the middle third layers is larger than $(12 n')^{-1}$
(instead of $1/6$, see page 1158 in \cite{GW00} for more details).
With these modifications, the proof of \cite{GW00} yields the
statement.

\end{proof}

\subsection{Upper estimate}
\label{inc:upper}
 We assume that there is one principal incipient horizon,
 if there were more, an analogous proof would apply.
As in Subsection \ref{sec:upper}, let us fix an $\varepsilon$, define the estimator environments
- one of them is the $2 \varepsilon$ neighbourhood of the principal incipient horizon ($
H^{2 \varepsilon}$), the
others have dimension at most $d-2$. The proportionality lemma implies that the $c$ portion
of a long enough flight is spent in one of the estimator environments. The $\mu$-measure
of such points for which this is not
$H^{2 \varepsilon}$ is $O(t^{-2})$ as in Case 1 of Subsection \ref{sec:upper}.\\
The essence of the proof is the following statement:
\begin{Lem}
\label{lemma:inc}
For a fixed $\varepsilon$ small enough,
\begin{eqnarray*}
&\lambda_{d} \times \lambda_{d-1}& (\{ x=(q,v) | q \in H^{2 \varepsilon},
\tau(x) > s, \Pi_Q \Phi_{[0,s]} x \subset H^{2 \varepsilon} \}) \\
&=&
\begin{cases}
    O(s^{-2}),  & 3 \leq d \leq 5\\
        O(s^{-2} \log s), & d = 6\\
	    O\left(s^\frac{2+d}{2-d}\right), & d > 6.
	      \end{cases}
\end{eqnarray*}
\end{Lem}

\begin{proof}
Denote by $V$ the $d-1$ dimensional hyperplane defining the incipient horizon.
Without loss of generality, we may assume that the origin is in this horizon,
that is $H=V$. Since $V$ is a lattice subspace, one can choose a lattice
vector $v_d \in \mathcal L \setminus V$ such that $V \cap \mathcal L$ and
$v_d$ generate $\mathcal L$. Since $\mathbb R / \mathcal L$ can be identified
with a parallelepiped generated by $v_1, \dots, v_d$ with $v_1, \dots v_{d-1}
\in V$, for every
$q \in Q \cap H^{2 \varepsilon}$, there is a unique decomposition
\[ q= q_V + q_W \]
with $q_V \in V$, $q_W \parallel v_d$ and $| q_W | <2 \varepsilon \cot \alpha$,
where $\alpha$ is the angle of $V$ and $v_d$. We also write
\[ v = v^{\parallel} + v^{\perp}, \]
where $v \in S^{d-1}, v^{\parallel} \in V, v^{\perp} \in V^{\perp}$.\\
The idea of the proof is reminiscent to that of Case 3 in Subsection \ref{sec:upper}.
If there is a long flight in $H^{2 \varepsilon}$, then $v$ is close to $V$.
Thus we can think of this trajectory as a long free flight in a $d-1$ dimensional
billiard. Note that here, the $d-1$ dimensional scatterer size can be arbitrary small,
since the trajectory is close to $V$. Thus a delicate analysis of this scatterer size,
and the upper estimation of (\ref{BGWbecsles}) are needed.\\
Let us chop the set of possible $q_W$'s and $v^{\perp}$'s into the following pieces:
\begin{eqnarray*}
&V_i = \{ v^{\perp} \in V^{\perp} | |v^{\perp} | \in
[2^{-i}, 2^{-i+1}) \} & i> \log s - \log 2 \varepsilon\\
&Q_j = \{ a v_d | |a| \in [2^{-j} \cot \alpha, 2^{-j+1} \cot \alpha) \} & j> - \log 2 \varepsilon.
\end{eqnarray*}
Accordingly, we write
\[
H_j = H^{2^{-j+1}} \setminus H^{2^{-j}}.
\]
Here, and also in the sequel, $\log$ always stands for $\log_2$.\\
Now assume that $v^{\perp} \in V_i$ and $q_W \in Q_j$ for some fixed $i,j$. We want to estimate
the $\lambda_{d-1} \times \lambda_{d-2}$ measure of parameters $q_V$, $v^{\parallel}$
with which $(q,v)$ is an element of the set
\[
Q_{long} = \{ x=(q,v) | q \in H^{2 \varepsilon},
\tau(x) > s, \Pi_Q \Phi_{[0,s]} x \subset H^{2 \varepsilon} \}.
\]
We can assume that the projection of
$q_W$ to $V^{\perp}$ and $v^{\perp}$ are
oppositely oriented. If they are not, a simpler version of
the forthcoming proof is applicable.\\
From now, we distinguish four cases.
\begin{itemize}
\item {\bf Case a}
%\paragraph{Case a}
$i<\frac{d}{d-2} \log s$ and
$j \leq i - \log s$.\\
In this case, there is a line segment of $\Pi_Q \Phi_{[0,s]} x$
of length at least $s/5$ spent in the strip $H_{j+1}$.\\
Note that for every $q \in H_{j+1}$,
the intersection of $Q$ and $q+V$ is a billiard configuration of dimension $d-1$. Further,
this billiard configuration is contained in a larger one, where there is only one spherical scatterer
of radius approximately $\sqrt{\kappa_{\max}^{-1} 2^{-j}}$. Indeed, there is at least one $d$ dimensional
scatterer touching $V$ from the appropriate side. If one takes the $d$ dimensional ball of radius
$\kappa_{\max}^{-1}$ touching $V$ in this point and considers the intersection of the ball and a close enough affine
hyperplane, obtains a $d-1$ dimensional ball of the desired radius (which is roughly the square root
of the distance of the hyperplanes). As in Case 3 of Subsection \ref{sec:upper}, by projecting
the previously obtained trajectory segment of length $s/5$ to the ``lower boundary of $H_{j+1}$''
(i.e. $\partial  H^{2^{-j}}$) we obtain a free flight of length at least $s/6$ (if $\varepsilon$
is small enough) in a $d-1$ dimensional billiard table with periodicity $\mathcal L \cap V$ and
one spherical scatterer of radius $\sqrt{\kappa_{\max}^{-1} 2^{-j}}$. Note that this mapping to the
lower dimensional billiard is simpler than that of Subsection \ref{sec:upper}, since
$V^{\perp}$ is one dimensional, thus the billiard configuration space in $q+V$ is
increasing as $q$ moves from $\partial H^{2 \varepsilon}$ to $V$ (the issue of moving scatterers
is simply absent). Observe that $i<\frac{d}{d-2} \log s$ and
$j \leq i - \log s$ imply $j \leq \frac{2}{d-2} \log s$ which yields that (\ref{BGWfeltetel}) is
satisfied by $t=s$, $r= \kappa_{\max}^{-1/2} 2^{-j/2}$ and $D = d-1$.
Thus the second part of (\ref{BGWbecsles}) implies that whenever $v^{\perp} \in V_i$ and $q_W \in Q_j$
are fixed, the $\lambda_{d-1} \times \lambda_{d-2}$ measure of parameters $q_V$, $v^{\parallel}$
with which $(q,v) \in Q_{long}$ is $O(s^{-1} 2^{j(d-2)/2})$.

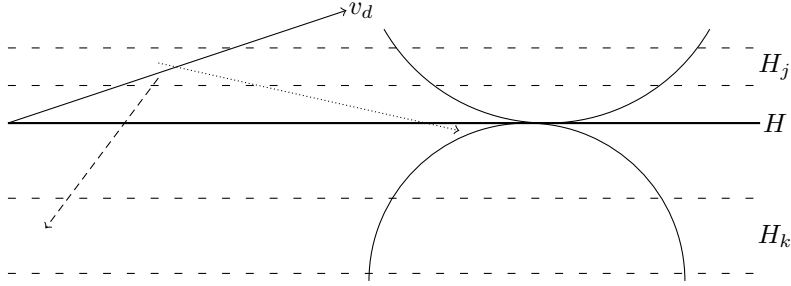
\begin{figure}
\centering
\begin{tikzpicture}

\draw[thick] (0,0) -- (10,0);
\draw[->] (0,0) -- (4.5,1.5);

\draw (5,1.25) arc (210: 330: 2.5);
\draw (9,-2.1) arc (00: 180: 2.1);

\draw[loosely dashed] (0,1) -- (10,1);
\draw[loosely dashed] (0,0.5) -- (10,0.5);

\draw[loosely dashed] (0,-1) -- (10,-1);
\draw[loosely dashed] (0,-2) -- (10,-2);

\draw[densely dotted, ->] (2,0.8) -- (6,-0.1);
\draw[densely dashed, ->] (2,0.6) -- (0.5, -1.4);

\draw (4.7,1.5) node {$v_d$}
(10.2,0.75) node {$H_j$}
(10.2, -1.5) node {$H_k$}
      (10.2, 0) node {$H$};

\end{tikzpicture}
\caption{$H^{2 \varepsilon}$ - a $d$ dimensional picture. Densely dotted
trajectory: $j<i-\log s$.
Densely dashed trajectory: $j \geq i- \log s$.}
\label{fig:V_H}
																					\end{figure}

\item {\bf Case b} $\frac{d}{d-2} \log s \leq i <
\frac{d+2}{d-2} \log s$ and
$j \leq \frac{2}{d-2} \log s$.\\
The same estimation as in Case a yields that
the $\lambda_{d-1} \times \lambda_{d-2}$ measure of parameters $q_V$, $v^{\parallel}$
with which $(q,v) \in Q_{long}$ is $O(s^{-1} 2^{j(d-2)/2})$.

\item {\bf Case c} $i<\frac{d}{d-2} \log s$ and
$j > i - \log s$.\\
Note that distance of $\Pi_{V^{\perp}} \Pi_Q x$
and $\Pi_{V^{\perp}} \Pi_Q \Phi_s x$ (here, $\Pi_{V^{\perp}}$
is the orthogonal projection to $V^{\perp}$) is at least $s2^{-i}$,
which is larger
than $2^{-j}$. Hence there is a $k$ such that a line segment of
$\Pi_Q \Phi_{[0,s]} x$ of length at least $s/8$ is spent in $H_k$ and
$2^{-k}$ is larger than $s2^{-i}/4$. Now using the same estimation as in Case a
in the strip $H_k$, one obtains that
the $\lambda_{d-1} \times \lambda_{d-2}$ measure of parameters $q_V$, $v^{\parallel}$
with which $(q,v) \in Q_{long}$ is $O(s^{-1-\frac{d-2}{2}} 2^{i(d-2)/2})$.

\item {\bf Case d} $\frac{d}{d-2} \log s \leq i <
\frac{d+2}{d-2} \log s$ and
$j > \frac{2}{d-2} \log s$, or $i \geq \frac{d+2}{d-2} \log s$.\\
In this case, we simply estimate the measure of the appropriate
parameters $q_V$, $v^{\parallel}$ by a constant.

\end{itemize}

Note that $\lambda_1( V_i) \sim 2^{-i}$ and $\lambda_1 (Q_j) = 2^{-j}$.
Taking into account this fact and the estimations of Cases a-d,
one obtains that $\mu (Q_{long})$ is bounded from above by some constant
times the following expression:
\begin{eqnarray*}
&&
\sum_{i=\log s - \log 2 \varepsilon}^{\frac{d}{d-2} \log s}
\left[
\left( \sum_{j= - \log 2 \varepsilon}^{i- \log s} 2^{-i} 2^{-j} s^{-1} 2^{j(d-2)/2}
\right) + 2^{-i} s2^{-i} s^{-1-\frac{d-2}{2}} 2^{i(d-2)/2}
\right]\\
&+& \sum_{i= \frac{d}{d-2} \log s}^{\frac{d+2}{d-2} \log s}
\left[
\left( \sum_{j= - \log 2 \varepsilon}^{\frac{2}{d-2} \log s} 2^{-i} 2^{-j} s^{-1} 2^{j(d-2)/2}
\right) + 2^{-i} s^{\frac{2}{2-d}}
\right] + s^{\frac{2+d}{2-d}}.
\end{eqnarray*}
An elementary computation shows that this is the same order of magnitude as
stated in the lemma.

\end{proof}

In order to finish the proof of the upper estimate, we need to bound the
measure of points $x=(q,v)$ for which $\tau(x) >t$ and
the proportionality lemma gives the
estimator environment $H^{2 \varepsilon}$. Observe that in this case,
the angle of $v$ and $V$ is necessarily smaller than $2 \varepsilon /t$.
The Lebesgue measure of points for which $q \in H^{2 \varepsilon}$
is bounded by the desired order of magnitude due to
Lemma \ref{lemma:inc}. Thus assume that $q \notin H^{2 \varepsilon}$.
For every such point $x=(q,v)$, there is a point
$\phi (x) = x_b = (q_b, v)$, which is the initial point of the
free flight segment in $H^{2 \varepsilon}$ (i.e.
$\exists s< (1-c) \tau(x): \Phi_s (x) = (q_b, v)$,
$q_b \in \partial H^{2 \varepsilon}$,
$\Pi_Q \Phi_{[s, s+ c\tau(x)]} x \subset H^{2 \varepsilon}$).
The proportionality
lemma also implies that for any such $x_b$,
\[ \lambda_1 (\phi^{-1} (x_b)) < \frac{1}{c}
\max \{ s: s<\tau(x_b), \Pi_Q \Phi_{[0,s]} x_b \subset H^{2 \varepsilon}\}.
\]
Thus, also using Lemma \ref{lemma:inc} (with $s=c t /2$), the integral
\[ \int_{\partial H^{2 \varepsilon} \times \{ v: \angle (v,V) < 2 \varepsilon /t
\}} \sin(\angle (v,V))
\lambda_1 (\phi^{-1} (x_b)) d \lambda_{d-1}(q_b) \times \lambda_{d-1} (v)
\]
can be bounded by the desired order of magnitude which yields the upper
estimate of Theorem \ref{tetel2}.

\subsection{Lower estimate}
\label{inc:lower}

Now, we prove the second part of Theorem \ref{tetel2}, which is a lower estimate
in the dispersing case. \\
In dimension $d \leq 5$,
the statement is straightforward, since
obviously there are horizons of codimension $2$ ``attached'' to the incipient horizon
(indeed, a hyperplane parallel to the incipient horizon and close to it,
intersects the scatterers in tiny convex bodies - approximate
ellipsoids - which depend continuously on the
distance of the hyperplanes). Then the same argument
used to prove (\ref{eq:F_H}) provides a subset of the phase space of measure
$O(t^{-2})$ consisting of points having free flight longer than $t$.\\
In dimension $d \geq 6$, we use a simplified version of the proof of Lemma
\ref{lemma:inc}. The main observation is that due to the lower bound on
the curvature, the scatterers touch the incipient horizon in finitely
many points (in $q_1, \dots q_{n'}$, say). Further, the intersection
of the scatterers and a hyperplane parallel to
the incipient horizon at distance $h$ from it, is contained in $n'$
spheres of radius $\sqrt{ \kappa_{\min}^{-1} h }$ centered at
$q_1, \dots q_{n'}$. Thus
in Cases a-d of Lemma
\ref{lemma:inc},
by such a choice of $i$ and $j$, where $s2^{-i} \approx
2^{-j}$, using the first part of (\ref{BGWbecsles})
instead of the second, one easily
obtains a lower bound of the same order of magnitude. In fact, for $d>6$,
only one pair of indices $(i,j)$ is enough. Namely, choose
\[ i= \lceil \frac{d}{d-2} \log s  \rceil  \]
and $j = \lceil i - \log s \rceil$. With this choice and the notation
$r= \sqrt{ \kappa_{\min}^{-1} s 2^{-i} }$, $s=t$,
(\ref{BGWfeltetel}) is fulfilled, hence the Lebesgue measure of points
$x=(q,v)$ with $v^{\perp} \in V_i$ and $q_W \in Q_j$ having
free flight longer than $s$ is at least some constant times
$2^{-i} 2^{-j}$, thus another constant times $s^{\frac{2+d}{2-d}}$.\\
In dimension $d = 6$, one needs to consider all indices $i$ with
$\log s - \log 2 \varepsilon < i
< 3/2 \log s - \log \kappa_{\min} $ and for a fix $i$, the index
 $j = i - \log s$. Similarly to the case $d>6$, the lower estimation
 of order $s^{-2} \log s$ follows.

\section{Examples}
\label{sec:ex}

 Equ. (35) of \cite{D12} provides the form of the limiting covariances for the super-diffusive limit of dispersing Lorentz processes assuming his Conjectures 1 and 3 hold. His derivation of Equ. (35) from the conjectures can be extended to the semi-dispersing case thus our Theorem 1 can be used. His Conjecture 3 is of dynamical nature and for clarity we briefly summarize what is known and what we expect in general. For brevity - beside \cite{D12} - we rely here on the works \cite{Y98,BT08} where, for instance, the complexity hypothesis is also used and the precise forms of exponential decay of correlations (EDC) and of the central limit theorem (CLT) are given.

\begin{itemize}
\item  \cite{BT08} For multidimensional ($d > 2$) dispersing billiards with finite horizon satisfying the complexity hypothesis, EDC and  CLT hold and the diffusivity covariance is given by Green-Kubo;
    \end{itemize}
    In formulating what we expect we do not pursue the highest generality and will be satisfied to restrict ourselves to ergodic cylindrical billiards (cf. \cite{SSz00}).
    \begin{itemize}
\item  {\bf Conjecture A} (Dynamical) For multidimensional ($d > 2$)  ergodic cylindrical billiards with strictly convex bases 1. without a principal horizon and 2. satisfying the complexity hypothesis, EDC and  CLT hold and the diffusivity covariance is given by Green-Kubo;
\item   {\bf Conjecture B} (Dynamical) For multidimensional ($d > 2$)  ergodic cylindrical billiards with strictly convex bases 1. with at least one principal horizon and 2. satisfying the complexity hypothesis, EDC and the super-diffusive limit statement with scaling $\sqrt{n \log n}$ or $\sqrt{t \log t}$ hold. (cf. \cite{SzV07,ChD09} for $d=2$).
\end{itemize}
 \paragraph{Example 1: Cylindrical billiard on $\mathbb T^3$.} (We note that this was the first semi-dispersing billiard whose ergodicity had been established (cf. \cite{KSSz89}).) We assume that on $\mathbb T^3$
we are given two nonintersecting cylindrical scatterers $C_1$ and $C_2$ - for simplicity - of equal radii $0 < r < 1/4$. Suppose that the generator of $C_i$ is parallel to the coordinate direction $e_i$, $i = 1, 2$ and the distances  between the two cylinders - in the coordinate direction $3$ - are $z$ and $w$. In this case we have two principal horizons of widths $z$ and $w$ parallel to the coordinate plane $(e_1, e_2)$ and super-diffusion is expected in the directions $e_1, e_2$ whereas regular one in the direction of the axis $e_3$.

\[
 {\mathcal D}_{11} =  {\mathcal D}_{22} = \frac{1}{4(1 - 2 r^2 \pi)}(z^2 + w^2)
 \]

 \[
 {\mathcal D}_{33} = 0
 \]

Of course, if - in the direction of the axis $e_3$ - we apply diffusive scaling, then the limiting covariance in that direction should again be given by the Green-Kubo formula.

 \paragraph{Example 2: Two hard balls of radii $1/(4 \sqrt 2) < r < 1/4$ on $\mathbb T^d$.} Under the complexity hypothesis it follows from  \cite{BT08} and from Theorem 1 that
for the super-diffusive limiting covariance $\mathcal D$ of the system,
$\mathcal D_{ij} = \delta_{ij} \mathcal D$, where
\begin{eqnarray*}
\mathcal D &=&
\sqrt 2 \frac{1}{1-|B_d|(2r)^d} \frac{|B_{d-1}|}{|S_{d-1}|} (d-1)(1-4r)^2
\\ &=&
\sqrt 2 \frac{1}{1-\frac{\pi^{d/2}}{\Gamma((d+2)/2)}(2r)^d}
\frac{\Gamma(d/2)}{\sqrt \pi \Gamma((d-1)/2)} (1-4r)^2.
\end{eqnarray*}
Here $B_d$ is the d-dimensional unit ball
and $S_{d-1}$ is its surface (cf. Equ. (37) of \cite{D12}).

\section{Concluding remarks}
\label{sec:rem}

\begin{enumerate}

\item Our methods also make possible to obtain the asymptotics of the free path length for cases when the maximal, but not  principal,  horizon(s) are incipient. We omitted the discussion for brevity.
\item In order to prove the above Conjecture B, a first step
could be determining the limiting joint distribution of $\tau$ and
the forthcoming free flight (i.e. $\tau \circ
\Phi_{\tau +}$, where
$\Phi_{\tau +}$ means that the velocity is the post-collisional one),
when $\tau$ is large (see also
Conjecture 3 in \cite{D12} or  in the
planar case \cite{B92} and \cite{SzV07}). Thus we formulate another conjecture.
\begin{itemize}
\item \textbf{Conjecture C} (Geometric) In a $d$ dimensional dispersing
billiard with at least one principal, non-incipient horizon, if $\tau$ is large,
then $\tau \circ \Phi_{\tau +}$ is typically of order
$\tau^{1/d}$.
\end{itemize}
Now we explain why we expect Conjecture C to be true.
Note that if $\tau(x)$ is
larger than some large $t$, then $x=(q,v)$ - with probability close to one - is such that $q$ is in a
principal horizon $H$, and the angle of $v$ and $V_H$ is roughly $1/t$. Further, the component
of $v$ in $V_H$ is uniformly distributed. After some time,
the free flight from $x$ reaches the boundary $\partial B_H
\times V_H$ of the horizon. Now we claim that the remaining time until the
collision is typically $t^{\frac{d-2}{d}}$, or in other words,
the distance of $\Pi_{V_H^{\perp}} \Pi_Q \Phi_{\tau(x)} x$ and
$B_H$ is roughly $t^{- 2/d}$. Indeed, in the hyperplane $q_h+V_H$ at distance
$h$ from $B_H +V_H$, there are $d-1$ dimensional scatterers
(approximate ellipsoids of bounded eccentricity due to the dispersing assumption)
of diameter $\sqrt h$. Thus (\ref{BGWbecsles}) yields that in this hyperplane,
a $\lambda_{d-1} \times \lambda_{d-2}$-typical
phase point does not collide until time $ht$ if and only if $ht << h^{\frac{2-d}{2}}$.
Now a similar argument used to prove Lemma \ref{lemma:inc}, implies that
typically the distance of $\Pi_{V_H^{\perp}} \Pi_Q \Phi_{\tau(x)} x$ and
$B_H$ is roughly $h=t^{- 2/d}$. Denote the post collisional velocity by $v'$.
We expect that the angle of $v'$ and $V_H$ is typically of order
$t^{- 1/d}$ which would provide Conjecture C.

\item Consider a $\mathbb{Z}^d$-periodic ($d \geq 2$)
 arrangement of balls of radii $r < 1/2$, and select a random direction
 and point (outside the balls). The dependence of the free flight function on $r$ will be denoted by $\tau_r$.
Bourgain, Golse and Wennberg, \cite{BGW98} initiated the study of the asymptotic behavior of $F_r(t) = \mu (\tau_r > t)$, when $r \searrow 0$. Since $\lim_{r \searrow 0} F_r(t) = 1$, the question makes sense in an appropriate scaling, only. Their main result showed that limit is non-trivial in the scaling $\mu (r^{d-1} \tau_r > t)$, only (known in statistical physics as the Boltzmann-Grad scaling). Then Marklof and Str\"ombergsson, \cite{MS10} could prove the existence of the limit for any lattice, any dimension, more general objects than spheres and also obtained further delicate results. In this limit dynamical questions can also be answered, and, in particular, Golse-Wennberg, \cite{GW00} showed that the limiting equation is not the classical Boltzmann one. Finally, Marklof and Str\"ombergsson, \cite{MS11} could prove that the limiting equation is a (second order Markov-) version of the linear Boltzmann equation. (As to a survey on these and related results see \cite{M10}.)
P\'olya's visibility problem is, in turn, related to the maximum of $\tau_r$,
when one erases a ball and chooses the initial point to be its
center (see \cite{P18} and \cite{K08}).
 \\
 On the other hand, for Lorentz processes with fixed configuration of scatterers, i. e. without the Boltzmann-Grad limit, Dettmann's conjectures elaborated and made those of Sanders (\cite{S05,S08}), based on computational observations, more precise.

Dettmann  observed that the constant in the tail asymptotics of the free path length in the Boltzmann-Grad limit of $\mathbb Z^d$-periodic spherical scatterers of radii $r \searrow 0$ (cf. Equ. (1.43) in \cite{MS11a}) coincides with the constant arising in his heuristic computation (cf. Equ. (31) in \cite{D12}) by taking  the large time limit of Theorem 1 and the limit $r \searrow 0$ in reversed order.  Marklof  has raised the intriguing question to prove this coincidence rigorously that would also require uniform estimate of remainder term in our Theorem 1.

 \item \cite{Sz08} also raised the problem of the limiting behavior of a quasi-periodic Lorentz process, for instance that of the Penrose-Lorentz one. As \cite{W12} points out the tail distribution of the free path length is exponential in random  Lorentz processes with non-intersecting scatterers whereas  - as we have seen - it is algebraic in the presence of horizons. The simulations of the author suggest that for a 1-dimensional quasi-periodic paradigm of the Lorentz process, this tail behavior  is not exponential. On the other hand,
    \cite{KS12} stresses that for the random non-intersecting Lorentz process one has normal diffusion and observes computationally three different regions for a 2-dimensional quasi-periodic Lorentz process showing super-diffusion, diffusion and subdiffusion.

\end{enumerate}

{\bf ACKNOWLEDGEMENTS.}
The authors thank Carl Dettmann for making them possible to read his manuscript during its preparation. Thanks are also due to Jens Marklof, Dave Sanders and to members of the Geometry Seminar at R\'enyi Institute for their valuable remarks. The support of the Hungarian National Foundation for Scientific
Research Grants No. K 71693 and K 104745 is gratefully acknowledged.
P. N.'s research was realized in the frames of T\'AMOP 4.2.4. 
A/1-11-1-2012-0001 "National Excellence Program - Elaborating 
and operating an inland student and researcher personal support system" 
The project was subsidized by the European Union and co-financed 
by the European Social Fund.


\begin{thebibliography}{99}

\bibitem[B92]{B92} P. M. Bleher. Statistical Properties of Two-Dimensional
  Periodic Lorentz Gas with Infinite Horizon. \textit{J. Stat.\ Phys.}
  \textbf{66} 1: 315--373 (1992).

\bibitem[BS81]{BS81} L. A. Bunimovich and Ya. G. Sinai.
Statistical properties of
Lorentz gas with periodic configuration of scatterers.
\textit{Comm. Math. Phys.} {\bf 78} 479--497, (1981).

\bibitem[BSCh91]{BSCh91} L. A. Bunimovich,  Ya. G. Sinai and N. I. Chernov.
Statistical properties of two
dimensional dispersing billiards.
\textit{Russian Math. Surveys} {\bf 46} 47--106, (1991).

\bibitem[BGW98]{BGW98} J. Bourgain, F. Golse and B. Wennberg.
On the Distribution of Free Path Lengths for the Periodic Lorentz Gas.
\textit{Comm. Math. Phys.} \textbf{190} 491--508, (1998).


\bibitem[BT08]{BT08}
P. B\'alint and I.P. T\'oth.
Exponential Decay of Correlations in Multi-dimensional Dispersing Billiards.
\emph{Annales Henri Poincar\'e} {\bf 9} 1309--1369, (2008).


\bibitem[ChD09]{ChD09} N. Chernov and D. Dolgopyat. Anomalous current in
    periodic Lorentz gases with infinite horizon. \emph{Russian
    Math. Surveys} \textbf{64} 73--124, (2009).

\bibitem[D12]{D12} C. P. Dettmann.
New horizons in multidimensional diffusion: The Lorentz gas and the Riemann Hypothesis,  \textit{J. Stat. Phys.} \textbf{146} 181--204, (2012).

\bibitem[GW00]{GW00} F. Golse and B. Wennberg.
On the Distribution of Free Path Lengths for the Periodic Lorentz Gas II.
\textit{ESAIM M2AN} \textbf{34} 1151--1163, (2000).

\bibitem[KS12]{KS12} A. S. Kraemer and D. P. Sanders.
Periodizing quasi-crystals: Anomalous diffusion in quasi-periodic systems, http://arxiv.org/abs/1206.1103

\bibitem[KSSz89]{KSSz89} A. Kr{\'a}mli, N. Sim{\'a}nyi and D. Sz\'asz.  Ergodic properties of semi-dispersing
billiards. I. Two cylindric scatterers in the 3-D torus.
\textit{Nonlinearity} \textbf{2} 311--326, (1989).

\bibitem[K08]{K08} CP Kruskal, The orchard visibility problem and some variants, \textit{J. Computer and System Sci.}, \textbf{74}, 587-597, (2008).

\bibitem[L05]{L05} H. Lorentz.
Le mouvement des \'electrons dans les m\'etaux.
\textit{Arch. N\'eerl.} \textbf{10} 336--371, (1905).

\bibitem[M10]{M10} J. Marklof, Kinetic transport in crystals, \textit{Proceedings of the XVI International Congress on Mathematical Physics},
Prague 2009, World Scientific, pp. 162-179, (2010).

\bibitem[MS10]{MS10} J. Marklof and A. Str\"ombergsson.
The distribution of free path lengths in the periodic Lorentz gas and related lattice point problems.
\textit{Annals of Mathematics} \textbf{172} 1949--2033, (2010).

\bibitem[MS11]{MS11} J. Marklof and A. Str\"ombergsson. The Boltzmann-Grad limit of the periodic Lorentz gas, \textit{Annals of Mathematics} \textbf{174} 225-298, (2011)

\bibitem[MS11a]{MS11a} J. Marklof and A. Str\"ombergsson. The periodic Lorentz gas in the Boltzmann-Grad limit: asymptotic estimates, \textit{GAFA Geometric and Functional Analysis}  \textbf{21} 560-647, (2011)

\bibitem[P18]{P18} G. P\'olya. Zahlentheoretisches und wahrscheinlichkeitstheoretisches über die Sichtweite im Walde, \textit{Arch. Math. Phys. Ser. 2}, \textbf{27}, 135-142, (1918).

    \bibitem[S05]{S05} D. P. Sanders, Deterministic Diffusion in Periodic Billiard Models, Thesis, U. of Warwick, pp. 204. (2005)  arXiv:0808.2252 [cond-mat.stat-mech]

\bibitem[S08]{S08} D. P. Sanders. Normal diffusion in crystal structures and higher-dimensional billiard models with gaps,
\textit{Phys. Rev. E} \textbf{78} 060101, (2008).

\bibitem[Sch68]{Sch68} W. Schmidt. Asymptotic formulae for point lattices of bounded determinant and subspaces of bounded height.
\textit{Duke Math. J.} \textbf{35} 327-339, (1968).

\bibitem[SSz00]{SSz00} N. Sim\'anyi and D. Sz\'asz.
Non-Integrability of Cylindric Billiards and Transitive Lie-group Actions,
\textit{Ergodic Theory and Dynamical Systems} {\bf 20} (2000), 593-610

\bibitem[Sz94]{Sz94} D. Sz\'asz. The K-Property of `Orthogonal' Cylindric Billiards.
\textit{Comm. Math. Phys.} \textbf{160} 581--597, (1994).

\bibitem[Sz08]{Sz08} D. Sz\'asz. Some challenges in the theory of (semi)-dispersing billiards.
\textit{Nonlinearity} \textbf{21} 187--193, (2008).

\bibitem[SzV07]{SzV07} D. Sz\'asz and T. Varj\'u. Limit Laws and Recurrence for
    the Planar Lorentz Process with Infinite Horizon. \emph{J.
    Stat. Phys.} \textbf{129} 59--80, (2007).

\bibitem[W12]{W12} B. Wennberg. Free path lengths in quasi crystals. http://arxiv.org/abs/1201.0450

\bibitem[Y98]{Y98} L. S. Young. Statistical properties of dynamical systems with some hyperbolicity.
\textit{Annals of Mathematics} 585--650, (1998).



\end{thebibliography}
\end{document}